\documentclass[11pt]{article}  
\usepackage{amsmath,amsthm}
\usepackage{amssymb}
\usepackage[margin=1.35in]{geometry}
\usepackage{bm,comment,hyperref}
\usepackage{natbib,enumitem}
\usepackage{graphicx,color}

\usepackage[plain,noend]{algorithm2e}

\newcommand{\diff}{\text{d}}

\newcommand{\vc}{\operatorname{vec}}

\newcommand{\SPD}{\operatorname{Sym}_{>0}}
\newcommand{\dd}{\text{d}}
\newcommand{\norm}[1]{\left\lVert#1\right\rVert}

\newcommand{\kbc}[1]{\textcolor{blue}{[KB: #1]}}

\newcommand{\fre}{\gamma_{\text{Fre}}}
\newcommand{\hatfre}{\hat \gamma_{\text{Fre}}}

\newtheorem{theorem}{Theorem}[section]
\newtheorem{corollary}[theorem]{Corollary}
\newtheorem{lemma}[theorem]{Lemma}
\newtheorem{proposition}{Proposition}[section]

\theoremstyle{definition}
\newtheorem{definition}[theorem]{Definition}
\theoremstyle{remark}
\newtheorem{example}{Example}[section]
\newtheorem{remark}{Remark}[section]

\makeatletter
\renewcommand{\algocf@captiontext}[2]{#1\algocf@typo. \AlCapFnt{}#2} 
\def\@algocf@capt@plain{top}
\renewcommand{\algocf@makecaption}[2]{%
  \addtolength{\hsize}{\algomargin}%
  \sbox\@tempboxa{\algocf@captiontext{#1}{#2}}%
  \ifdim\wd\@tempboxa >\hsize
    \hskip .5\algomargin%
    \parbox[t]{\hsize}{\algocf@captiontext{#1}{#2}}
  \else%
    \global\@minipagefalse%
    \hbox to\hsize{\box\@tempboxa}
  \fi%
  \addtolength{\hsize}{-\algomargin}%
}
\makeatother

\numberwithin{equation}{section}

\title{Rolled Gaussian process models for curves on manifolds}

\usepackage{bigfoot}

\DeclareNewFootnote{AAffil}[arabic]
\DeclareNewFootnote{ANote}[fnsymbol]

\usepackage{etoolbox}
\makeatletter
\patchcmd\maketitle{\def\@makefnmark{\rlap{\@textsuperscript{\normalfont\@thefnmark}}}}{}{}{}
\makeatother

\makeatletter
\def\thanksAAffil#1{
  \footnotemarkAAffil\protected@xdef\@thanks{\@thanks%
        \protect\footnotetextAAffil[\the \c@footnoteAAffil]{#1}}%
}
\def\thanksANote#1{%
  \footnotemarkANote%
  \protected@xdef\@thanks{\@thanks%
        \protect\footnotetextANote[\the \c@footnoteANote]{#1}}%
}
\makeatother

\author{
  S. P. Preston%
  \thanksAAffil{School of Mathematical Sciences, University of Nottingham, UK}%
    $^{,}$
   \thanksANote{simon.preston@nottingham.ac.uk}
  , %
  K. Bharath%
  \footnotemarkAAffil[1]
  , %
  P. C. L\'opez-Custodio%
  \thanksAAffil{Computer Science Department, Nottingham Trent University, UK}%
  ,
  A. Kume
  \thanksAAffil{School of Mathematics, Statistics and Actuarial Science, University of Kent, UK}
}
\date{}
\begin{document}

\maketitle

\begin{abstract}
Given a planar curve, imagine rolling a sphere along that curve without slipping or twisting, and by this means tracing out a curve on the sphere. It is well known that such a rolling operation induces a local isometry between the sphere and the plane so that the two curves uniquely determine each other, and moreover, the operation extends to a general class of manifolds in any dimension. We use rolling to construct an analogue of a Gaussian process on a manifold starting from a Euclidean Gaussian process with mean $m$ and covariance $K$, and refer to it as a rolled Gaussian process parameterized by $m$ and $K$. The resulting model is generative, and is amenable to statistical inference given data as curves on a manifold. We identify conditions on the manifold under which the rolling of $m$ equals the Fr\'echet mean of the rolled Gaussian process, propose computationally simple estimators of $m$ and $K$, and derive their rates of convergence. 
We illustrate with examples on the unit sphere, symmetric positive-definite matrices, and with a robotics application involving 3D orientations.
\end{abstract}


\section{Introduction}


Many modern applications generate data where the statistical unit is a curve $t\mapsto x(t) \in M$ on a manifold $M$. Examples include data as trajectories on spheres \citep{SSKS,zhang2018phase}, shape spaces \citep{chevallier2021coherent}, the space of covariance matrices \citep{zhang2018rate}, and the space of probability densities \citep{chen2017modelling}. 
Despite prevalence of such data, there is a dearth of simple generative statistical models that are suitable, and amenable for inference, for discretely observed data curves.  

For curves in a Euclidean space, $\mathbb{R}^d$, Gaussian processes offer an appealing model due to their characterization by a mean and a covariance function, computational tractability, and closure under marginalization and conditioning.  However, Gaussian processes rely on covariance operators defined using the global vector space structure of $\mathbb R^d$, which is absent on a $d$-dimensional nonlinear $M$. Having an analogue of a Gaussian process on $M$ will facilitate modelling and inference of manifold-valued curve data. A key objective of this paper is to define such a process. 


A strategy to address manifold nonlinearity is to consider a mapping between $M$ and $\mathbb R^d$, enabling a curve on $M$ to be identified with a curve on $\mathbb R^d$ and vice versa; these are respectively ``flattening'' and ``unflattening'' maps. Pairing a generative model for curves in $\mathbb R^d $ with suitable unflattening maps provides a generative model for curves on $M$. In the other direction, for inference given observed data that are curves on $M$, the flattening maps facilitate practical computations and statistical inference for the model parameters. We follow this strategy: for the model in $\mathbb R^d$ we use the Gaussian process, for unflattening we use a pair of maps termed ``rolling'' and ``wrapping'', and for flattening we use a reverse pair of maps termed ``unrolling'' and ``unwrapping''.  We define these four maps precisely below. 

The four maps are compatible with the intrinsic geometry of $M$, and are optimal in a sense clarified below in not introducing distortions. This is in contrast to identifying $\mathbb R^d$ with a tangent space at a single fixed point on $M$, and using tangent space projections between the tangent space and $M$, and vice versa, which causes distortions. Our maps provide a natural way to define a covariance, and hence a Gaussian process-type model, on $M$, starting with one in $\mathbb R^d$.

\subsection{Related work and contributions}


``Rolling'' of a curve in $\mathbb R^d$ onto $M$, and its reverse operation, ``unrolling'', originate from mathematical geometry, where they are known as \emph{antidevelopment} and \emph{development} respectively; see \citet{kobayashi1996foundations} and \citet{sharpe2000differential}. \citet{jupp1987fitting} first exploited the idea in statistics, to define smoothing curves for time-indexed data points on the unit sphere. \citet{kume2007shape} adapted their technique to shape spaces, and later
\citet{kim2021smoothing} generalized to Riemannian manifolds. These works, however, did not develop generative statistical models. 

Stochastic development and the definition of an intrinsic Brownian motion on $M$, \citep[e.g.][]{hsu2002stochastic} and stochastic differential equations in $M$, can be used to define generative diffusion models for curves on $M$ \citep[e.g.][]{sommer2019infinitesimal,sommer2020probabilistic}. However, the Markovian property of the models does not allow for rich covariance structures, and computational challenges abound: exact simulation is infeasible and numerical approximations are costly; likelihood functions are difficult to compute; and the infinitesimal formulation is not well-suited to modelling discretely observed curves.

Outside of the stochastic-differential-equation-based models, the recent papers have broadly focused on two different aspects of statistical analysis of data curves on $M$: (a) descriptive, where the focus has been on development and theoretical analysis of principal component analysis (PCA) with various notions of covariance of a random $M$-valued curve $x$ along a base curve $\gamma$, usually the Fr\'echet mean curve; (b) model-based, by wrapping a Gaussian vector field along a $\gamma$, but for which the practical and theoretical properties of estimators of model parameters are unclear. Relevant works with focus on aspect (a) include the extrinsic approach on embedded submanifolds $M$ by \citet{dai2018principal}, and intrinsic approaches by \cite{lin2019intrinsic} and \cite{shao2022intrinsic} with covariance defined using the family of tangent spaces along $\gamma$. Generality of the setting needed to develop the asymptotic theory of PCA in these papers impose restrictive assumptions that are usually difficult to verify in practice; for example, support of the distribution of $x$ is effectively restricted, either explicitly or via a condition on the Fr\'echet functional for $x$. In this setting, see also \citep{dubey2020functional, dubey2022modeling} for curve methodology in general metric space, not explicity using the differentiable structure of a manifold, $M$. 
Works related to (b) include papers by \cite{mallasto2018wrapped} and \cite{wang2024intrinsic} which either pre-specify $\gamma$ or estimate it separately from the model, and \cite{liu2024wrapped} which is focused on the regression setting in the presence of Euclidean covariates. In these, theoretical support for the estimators is limited and convergence rates are absent. 

Common to all papers that have focused on aspect (a) or (b) is a random vector field $V^\gamma$ along $\gamma$ with covariance defined with respect to the \emph{family of tangent spaces} along $\gamma$. Reconciling such abstract definitions of covariance \citep[e.g.][]{shao2022intrinsic,lin2019intrinsic} with the practicality of needing coordinates for computation is a key complication when developing a unified modelling framework with theoretical guarantees. 

The distinguishing feature of our approach lies in the \emph{joint} specification of the base curve $\gamma$ and the vector field $t\mapsto V^\gamma(t)$ along $\gamma$ starting from a Euclidean Gaussian process in a \emph{single tangent space}. The rolling procedure results in a $\gamma$ and a covariant Gaussian vector field $V^\gamma$ such that $V^\gamma(s)$ and $V^\gamma(t)$ are related by parallel transport along $\gamma$. 
Accordingly, our main contributions are as follows.
\begin{itemize}[leftmargin=*]
\itemsep 0em
    \item We use rolling to define an analogue of a Gaussian process $x$ on $M$ defined via the rolling of a Gaussian process $z$ with mean $m$ and covariance $K$ from a single tangent space, that is compatible with the intrinsic geometry of $M$. The model is equivariant to the choice of tangent space coordinates (Proposition \ref{prop:equivariance}). 
    \item We determine conditions on $M$ that ensure that the rolling of $m$ equals the Fr\'echet mean of $x$ (Theorem \ref{thm:frechet:mean}). From this, using  the reverse operations of unrolling and unwrapping onto a single tangent space, we propose simple estimators of $m$ and $K$ without having to choose coordinates for the family of tangent spaces along a curve. 
    \item We derive rates of convergence of the estimators for fully observed curves (Theorem \ref{thm:rates}) that match the optimal rates in the Euclidean setting under assumptions that are easy to verify and satisfied by several manifolds of interest in statistics. We also derive two results that are new in the literature: (i) Aided by the properties of the Gaussian process $z$, we prove $C^1$ convergence of the sample Fr\'echet mean curve to its population counterpart (Theorem \ref{thm:C1_convergence}), which improves upon the present state-of-the-art uniform convergence \citep{dai2018principal, lin2019intrinsic}; (ii) a holonomy bound with explicit cuvature-dependent constants for triangles in $M$, which is then used in deriving convergence rates for the estimators (Lemma \ref{lemma:holonomy}). 
    \item We consider a simple parametric rolled Gaussian model defined using basis functions in \S\ref{sec:parametric:model:for:curves:on:Rd} that is tailored for discretely observed curves, and
    in \S\ref{sec:robot:SO3} employ it on data curves of three-dimensional rotations  arising from a novel robotics application. 
\end{itemize}
Appendices A and B contain relevant definitions and computational details. Technical lemmas and proofs of all results are presented in Appendices C and D. 

\section{Preliminaries}
\label{sec:prelim}
We record relevant definitions from Riemannian geometry and refer to standard references \citep[e.g][]{petersen2016symmetric, lee2018introduction} for details. 

Let $M$ be a complete connected Riemannian manifold of dimension $d$. Denote by $T_pM$ the tangent space at each point $p\in M$ and by  $TM=\{T_p M:p\in M \}$ the collection of tangent spaces, known as the tangent bundle. On the tangent space $T_p M$ at each point $p$,  we can define an inner product $\langle\cdot,\cdot\rangle_p:T_p M\times T_p M\rightarrow \mathbb R_{\geq 0}$ with induced norm $\|\cdot\|_p$; the collection $\{\langle\cdot,\cdot\rangle_p:p\in M\}$ is referred to as a Riemannian metric. The Riemannian metric varies smoothly along the manifold and allows us to measure distances, volumes, and angles. The operator norm for maps $B:T_pM \to T_pM$ is defined as $\|B\|_{\text{op}}:=\sup\{\|Bv\|_p:v \in T_pM, \|v\|_p=1\}$. 

For a curve $\gamma:[0,1]\rightarrow M$, with image $\gamma([0,1]) \subset M$, its velocity at $t$ is $\dot\gamma(t) := \mathrm{d} \gamma(t) \slash \mathrm{d} t\in T_{\gamma(t)}M$, and its length $L(\gamma)$ is the scalar $\int_0^1 \|
\dot{\gamma}(t)\|^{1/2}_{\gamma(t)} \diff t$.
For all curves in this paper, the parameter domain is the unit interval, $[0,1]$. Often, but not always, $t$ is interpreted as time; we refer to it thus in the following. Depending on the context, we will denote a curve by both $\gamma$ or $\gamma(t)$. 

Vector fields are smooth maps $X:M \to TM$ that map a point $p \in M$ to a vector $X(p) \in T_pM$. 
The Riemannian metric defines a unique affine connection called the Levi-Civita connection $\nabla$, known as the covariant derivative, which describes how vectors in tangent spaces change as one moves along a curve $\gamma$ from one point to another so that $\nabla_XY$ is the covariant derivative of a vector field $Y$ with respect to another $X$. 
We use $\nabla_p X(p) \in T_pM$ to denote the covariant derivative of $\nabla_XX$ at the point $p$. Similarly, $\nabla_p^2X(p)$ is the second covariant derivative, or the Hessian, of the vector field $X$ at $p$. 

A vector field $t \mapsto V(t) \in T_{\gamma(t)}M$ along a curve $\gamma$ is said to be parallel if $\nabla_{\dot{\gamma}(t)} V(t)=0$ for all $t$. Given a curve $\gamma:[0,1] \to M$ and $u \in T_{\gamma(t_0)}M$ for $t_0 \in [0,1]$, there exists a unique parallel vector field $V$ such that $V(t_0)=u$.
Vectors in different tangent spaces are identified using \emph{parallel transport} along a curve $\gamma$: the parallel transport of $u \in T_{\gamma(s)} M$ along $\gamma$ to $T_{\gamma(t)} M$ is $V(t)$, where $V$ is the unique parallel vector field determined by $V(s)=u$. This engenders a linear isometry $P_{s \rightarrow t}^\gamma: T_{\gamma(s)}M \to T_{\gamma(t)}M, \quad P_{t \leftarrow s}^\gamma(u)=V(t)$, between tangent spaces at points along $\gamma$, such that $\langle u,w \rangle_{\gamma(s)}=\langle P_{s \rightarrow t}^\gamma(u), P_{s \rightarrow t}^\gamma(w) \rangle_{\gamma(t)}$ for every $u,w \in T_{\gamma(s)} M$. Its inverse is denoted by $P_{s \leftarrow t}^\gamma$. 

For a function $f:M \to \mathbb R$,  we use $\nabla_p f(p)$ to denote the gradient of $f$ at $p$, which is defined implicity via the relation $\text{d}f(x)(v)=\langle \nabla_pf(p),v\rangle_p$, where $\text{d}f(x):T_pM \to \mathbb R$ is the directional derivative of $f$ at $p$ along $v$. The Hessian of $f$ at $p$ then is the bilinear map $\nabla^2_pf(p):T_pM \times T_pM \to \mathbb R$ with the corresponding Hessian operator $H(p):T_pM \to T_pM$. For $t \in [0,1]$, we abuse notation and also denote as $\nabla_t f(c(t))$ and $\nabla^2_t f(c(t))$, the gradient and Hessian of $f$ at $c(t)=p$, where $c:[0,1] \to M$ is a curve. The notation is helpful in representing derivatives of a function $f:M \times [0,1] \to \mathbb R, (p,t) \mapsto f(p,t)$ with respect to either of the arguments.

Curves $\gamma$ with zero acceleration such that $\nabla_{\dot{\gamma}(t)}\dot{\gamma}(t) =0$ for all $t$ are called \emph{geodesics}. The distance $\rho$ between points $p$ and $q$ is the length of a segment of a geodesic curve connecting the two, known as the minimal geodesic: $\rho(p,q)=\inf \{L(\gamma)| \gamma:[0,1] \to M;\gamma(0)=p,\gamma(1)=q\}$. 

 Given a point $p$ and a geodesic $\gamma$ with $\gamma(0)=p$, a cut point of $p$ is defined as the point $\gamma(t_0)$ such that $\gamma$ is a minimal geodesic on the interval $[0,t_0]$ but fails to be for $t>t_0$. The set of cut points of geodesics starting at $p$ is its \emph{cut locus} $\mathcal C(p)$, and we also use $\mathcal C(A)=\{\mathcal C(p): p \in A \subseteq M\}$ to denote the the cut locus of a set. The injectivity radius $i(p)\in[0,\infty]$ of $p$ is its distance to $\mathcal C(p)$, and the injectivity radius $i(A)$ of $A \subseteq M$ is the infimum of the injectivity radii of points in $M$. 

The next two tools are necessary for moving between $M$ and its tangent spaces. Geodesics on $M$ are used to define the \emph{exponential map} $\exp:TM\to M$, which maps a point $(p,v)$ in the tangent bundle to a point on $M$; in particular, when restricted to $T_pM$ it is defined via a geodesic $\gamma:[0,1] \to M$ starting at $\gamma(0)=p$ with initial velocity $\dot \gamma(0)=v$ as $\exp_p:T_p M \to M, \exp_p(v)=\gamma(1)$. From the Hopf-Rinow theorem, on a complete manifold the exponential map is surjective. On an star shape region around the origin in $T_pM$ this map is a diffeomorphism onto its image outside of the cut locus of $p$, and a well-defined inverse $\exp^{-1}_p:M \to T_p M$ exists, known as the \emph{inverse exponential} or logarithm map, that maps a point on $M$ outside of $\mathcal C(p)$ to $T_pM$; thus for any $q$ outside of $\mathcal C(p)$, $\rho(p,q)=\|\exp_p^{-1}(q)\|_p$. Relatedly, the \emph{tangent cut locus} $\mathcal T(p) \subset T_pM$ of $p \in M$, is defined by the relation $\exp_p(\mathcal T(p))=\mathcal C(p)$.

There are many notions of curvature of a Riemannian manifold. We will mainly be concerned with \emph{sectional curvature} at a point $p$, $\sec{(p)}$, defined to be the Gaussian curvature at $p$ of the two-dimensional surface swept out by the set of all geodesics starting at $p$ with initial velocities lying in the two-dimensional subspace of $T_pM$ spanned by two linear independent vectors. For $A \subset M$, the notation $\sec{(A)}$ represents $\{\sec(p):p \in M\}$ so that $k_0 \leq \sec{(A)} \leq k_1$ implies that the sectional curvature of every $p \in A$ is upper and lower bounded by $k_1$ and $k_0$, respectively. Throughout, we will simply use curvature to mean sectional curvature, and describe a manifold $M$ as positively (non-positively) curved if its curvature at every point is positive (non-positive). Non-positively curved $M$ have empty cut loci with $i(M)=\infty$ so that $\exp_p:T_p M \to M$ is a global diffeomorphism at every $p$. 

The \emph{Fr\'echet mean} of a probability measure $\nu$ on $M$ is defined as the minimizer of Fr\'echet functional
\begin{equation}
   y \mapsto  \int_M\rho^2(y,x) \, \text{d}\nu (x); 
\label{Eqn:FM_definition}
\end{equation}
sufficient conditions that ensure existence and uniqueness of the Fr\'echet mean on manifolds are well-known \citep[e.g.][]{afsari2011riemannian} and are assumed henceforth. When $\nu$ is the empirical measure of a sample of points in $M$, the minimizer of the Fr\'echet functional with respect to the empirical measure is the \emph{sample Fr\'echet mean}. Under suitable conditions the sample Fr\'echet mean converges almost surely to its population counterpart \citep{ziezold1977expected}. Abusing terminology, we also speak of the Fr\'echet mean of a random variable $x$ where $x \sim \nu$; context will disambiguate the two. For a random curve $x:[0,1] \to M$ the \emph{population Fr\'echet mean curve}, $\gamma_\text{Fre}:[0,1] \to M$, is defined as the Fr\'echet mean of $x$ pointwise in $t$; that is, minimizing 
\begin{equation}
  p \mapsto F(p,t) := \mathbb{E}[\rho^2(p,x(t))],
\label{eq:Fre_population}
\end{equation}
For a sample of curves, $\{x_1, \ldots x_n\}$, the \emph{sample Fr\'echet mean curve} is  $\hat{\gamma}_\text{Fre}:[0,1] \to M$, minimizing
\begin{equation}
 p \mapsto F_n(p,t) := \frac{1}{n} \sum_{i=1}^n \rho^2(p,x_i(t)).
\label{eq:Fre_mean_curve}
\end{equation}
We assume that $\fre$ exists and is unique for the rolled Gaussian model we propose (assumption A1). No such assumption will be made about its sample counterpart. 

Two types of manifolds are of particular interest in this work, especially when deriving the asymptotic properties of estimators of the model parameters: Hadamard manifolds and symmetric spaces. Simply connected, complete manifolds with non-positive sectional curvatures everywhere are known as \emph{Hadamard manifolds}. 
A \emph{symmetric space} is a homogeneous manifold $M$ in which for every $p \in M$ there is a unique isometry $\sigma_p:M \to M$ such that on $T_pM$ its differential,  $\text{d}\sigma_p=-\text{id}$, the negative of the identity map; thus, $\sigma_p$ reverses geodesics $\gamma$ through $p$ so that if $\gamma(0)=p$, then $\sigma_p(\gamma(t))=\gamma(-t)$. Symmetric spaces are classified as being of compact type or non-compact type, wherein the former is positively curved while the latter is non-positively curved.
A function $g:M \to \mathbb R$ on a symmetric manifold $M$ is said to be \emph{symmetric} if $g(p)=g(\sigma_p(p))$ for every $p \in M$. Relatedly, a distribution $\nu$ on $M$ is said to possess \emph{geodesic symmetry} about $ p\in M$, when $\nu=(\exp_{p})_\# \lambda$, where $\lambda$ is a mean-zero distribution on $T_pM$ with Lebesgue density $f$ that satisfies $f(v)=f(-v)$ for every $v \in T_pM$. The measure $\nu$ on $M$ defined as a pushforward under $\exp_p$ of a measure on $T_pM$ is commonly referred to as a wrapped distribution on $M$.




\section{The unrolling, rolling, unwrapping, and wrapping maps} 
\label{sec:four:maps}

Prior to defining the four maps for curves, we
introduce some preparatory notation and definitions.  The \emph{unrolling into the tangent space of $\gamma(0)$} of a piecewise smooth curve $\gamma:[0,1] \to M$ is the unique curve $\gamma^{\downarrow}:[0,1] \to T_{\gamma(0)}(M)$ with $\gamma^{\downarrow}(0)=0$, determined by the differential equation
\begin{equation} \label{eqn:unrolling:defn}
\dot{\gamma}^{\downarrow}(t)= P^\gamma_{0 \leftarrow t} \dot{\gamma}(t).
\end{equation}
Conversely, $\gamma$ on $M$ is determined by $\gamma^\downarrow$, and we refer thus to  $\gamma$ as the \emph{rolling} of the curve 
$\gamma^\downarrow$ from $T_{\gamma(0)}(M)$ onto $M$. Uniqueness of the parallel transport maps along $\gamma$ in \eqref{eqn:unrolling:defn} \citep[Theorem 4.32, ][]{lee2018introduction} ensures that the unrolling and rolling operations are well-defined, and are inverses of each other. Figure \ref{fig:roll_unroll_piecewise_geodesic}(a) provides an illustration of the unrolling of a piecewise geodesic curve on $\mathbb S^2$ onto the tangent space of its initial point. The process of unrolling has an intuitive physical analogy \citep{jupp1987fitting}: imagine that the curve $\gamma:[0,1] \to \mathbb S^2$ on the unit 2-sphere $\mathbb S^2$ is drawn in wet ink; its unrolling onto the plane $\mathbb R^2$, when identified with $T_{\gamma(0)}M$, is the trace of the wet ink left on $\mathbb R^2$ upon rolling $\mathbb S^2$ along $\gamma$ on $\mathbb R^2$ without slipping. Slipping is avoided by use of the parallel transport maps, which ensure that the tangent vectors $\dot \gamma(t)$ and $\dot \gamma^\downarrow(t)$ line up at each $t$. When $\mathbb S^2$ is viewed as an embedded submanifold of $\mathbb R^3$, the extra dimension demands that in addition to no slipping one also avoids ``twisting" \citep{sharpe2000differential}. Thus, 
it is appropriate to view $\gamma$ and $\gamma^\downarrow$ as two equivalent copies of the same curve. 


For piecewise geodesic curves, as in the example in Figure \ref{fig:roll_unroll_piecewise_geodesic}, a key feature of unrolling, and hence also of the rolling, is that segment lengths and angles between consecutive segments along $\gamma$ are preserved in $\gamma^\downarrow$, and vice versa. In contrast, a global flattening strategy by projecting the entire curve $\gamma$ onto $T_{\gamma(0)}M$ via the inverse exponential map would distort these angles and lengths in its 
image $\exp^{-1}_{\gamma(0)}(\gamma)$ relative those in $\gamma$. There is thus no such distortion when unrolling. For modelling purposes, therefore, a mean curve on $M$ may be conveniently defined as the rolling of a curve prescribed on a tangent space.

\begin{figure}[!t]
\begin{center}
\includegraphics[scale=0.35, trim={0cm 2cm 0cm 0.5cm},clip]{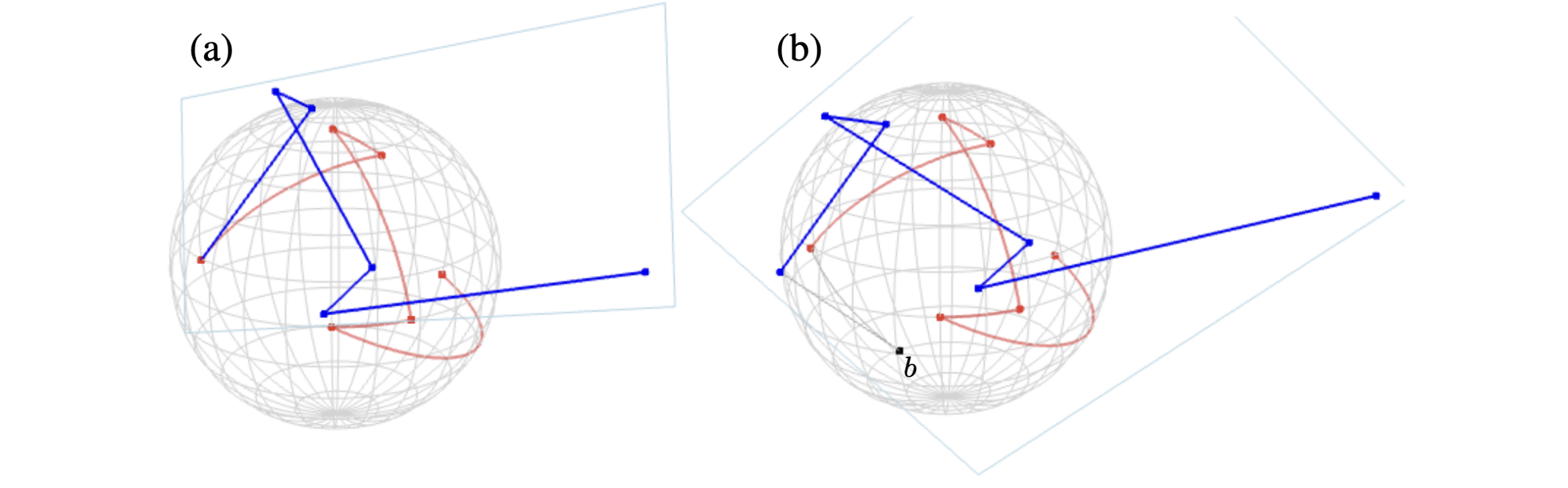}
\caption{A piecewise geodesic curve $\gamma$ on $\mathbb S^2$, red, and its unrolling, blue, into a piecewise linear curve; (a) shows the unrolling $\gamma^{\downarrow}$ into the tangent space $T_{\gamma(0)}\mathbb S^2$ at the initial point of the curve, and (b) shows the unrolling $\gamma_b^{\downarrow}$, incorporating a translation shown in grey, into the tangent space, $T_{b}$, at a prescribed point, $b$. Unrolling, and its inverse, rolling, preserve consecutive inter-point distances on $\gamma$ and $\gamma^\downarrow$.}
\label{fig:roll_unroll_piecewise_geodesic}
\end{center}
\end{figure}


To define a suitable covariance, we require a notion, compatible with rolling/unrolling, of how a curve $x$ on $M$ deviates 
relative to another fixed curve $\gamma$ on $M$. To begin with, we consider the deviation of a \emph{point}, $x$, that is associated with with a particular $t \in [0,1]$. Accordingly, relative to $\gamma(t)$, the \emph{unwrapping of point} $x$, assumed to be outside the cut locus $\mathcal C\{\gamma(t)\}$ of $\gamma(t)$, onto $T_{\gamma(0)}M$ defined as
\begin{align}
(t,x) \mapsto P^\gamma_{0 \leftarrow t}
\left\{ \exp^{-1}_{\gamma(t)} (x) \right\} \in T_{\gamma(0)}M.
\label{eqn:unwrap}
\end{align}
The reverse operation, for a point $y$ on $T_{\gamma(0)}M$ associated with $t \in [0,1]$ is the \emph{wrapping of point} $y \in T_{\gamma(0)}M$ onto $M$ with respect to the curve $\gamma$:
\begin{align}
(t, y) \mapsto \exp_{\gamma(t)} \left( P^\gamma_{t \leftarrow 0} y \right) \in M. 
\label{eqn:wrap}
\end{align}
Given a curve $\gamma$ on $M$, wrapping and unwrapping operations are inverses of each other outside of $\mathcal C(\gamma(0))$. 
\begin{definition}
 \label{def:cutlocus}
 Two curves $x$ and $y$ are said to be ``outside the cut loci of each other'' if for every $t$, $x(t)$ lies outside of $\mathcal{C}\{y(t)\}$, the cut locus of $y(t)$. 
 \end{definition}

It is possible to wrap an entire curve onto $M$, and not just a point, with respect to a fixed curve $\gamma$, by applying \eqref{eqn:wrap} pointwise for every $t \in [0,1]$. The same applies to unwrapping of a curve $x$ with respect to $\gamma$ using \eqref{eqn:unwrap}, when $x$ and $\gamma$ are outside the cut loci of each other. 
Consequently, we are hence in a position to define as follows the four maps central to this paper, which map curves from a manifold to a tangent space and vice versa.

These four maps are generalizations of those defined in \citet{kim2021smoothing} in several respects: first, rather than being defined for points, all of them map curves to curves; second, they are with respect to general curves, not necessarily piecewise geodesic curves; and third, we include a ``translation'', via a parallel transport operation along a unique geodesic between $b$ and $\gamma(0)$, such that the ``flat space'' in which we work is the tangent space $T_bM$ at some prescribed fixed point $b \in M$---see for example Figure \ref{fig:roll_unroll_piecewise_geodesic}(b). This proves important later, when we take $\gamma$ to play the role of the mean curve and seek to estimate it from data; the translation to $T_bM$ means that computations are in a fixed tangent space rather than a tangent space that depends on parameters being estimated.



\begin{definition}[The four maps]
  \label{defs}
Let $\mathcal F$ and $\mathcal F_b$ be the space of piecewise smooth curves on $M$ and $T_bM$, both with parameter domain $[0,1]$, where $b \in M \backslash \mathcal C(\gamma(0))$. Let $\gamma \in \mathcal F$, and let $c:[0,1] \to M$ be a geodesic from $b$ to $\gamma(0)$. 
 \begin{enumerate}
 \itemsep 1em
     \item [(a)]The \emph{unrolling} map $\mathcal F \to \mathcal F_b$ transforms a curve $\gamma$ into a curve $\gamma^\downarrow_b$, with 
     ${\gamma}^{\downarrow}_b(0) = \exp^{-1}_b\{ \gamma (0) \}$, satisfying 
\begin{equation} \label{eq:translated:unrolling}
\dot{\gamma}^{\downarrow}_b(t) 
=P^c_{0 \leftarrow 1}P^\gamma_{0 \leftarrow t} \dot{\gamma}(t).
\end{equation}

\item [(b)] The \emph{rolling} map $\mathcal F_b \to \mathcal F$ is the unique inverse of the unrolling map.
\item [(c)]
The \emph{unwrapping} map $\mathcal F \to \mathcal F_b$ transforms a curve $x$, with respect to another curve $\gamma \in \mathcal F$, into the curve 
\begin{align}
  x_b^{\downarrow \gamma}(t) = \gamma^\downarrow_b(t) + 
 P^c_{0 \leftarrow 1}  P^\gamma_{0 \leftarrow t}
\left\{ \exp^{-1}_{\gamma(t)} (x(t)) \right\},
\label{eqn:translated:unwrapping}
\end{align}
whenever $x$ is outside the cut loci of $\gamma$.

\item [(d)]
The \emph{wrapping} map $\mathcal F_b \to \mathcal F$ transforms a curve $y_b$, with respect to another curve $\gamma \in \mathcal F$, into the curve 
\begin{align}
y_b^{\uparrow \gamma}(t) = \exp_{\gamma(t)} \left[ P^\gamma_{t \leftarrow 0} P^c_{1 \leftarrow 0}  \left\{ y_b(t) - \gamma^\downarrow_b(t) \right\}\right].
\label{eqn:translated:wrapping}
\end{align}


\end{enumerate}
\end{definition}

\begin{remark}
Unwrapping of a curve $x$ on $M$ with respect to $\gamma$ on $M$ can be viewed as a two-step process: (i) obtain the vector field along $\gamma$ pointing towards $x$ via the inverse exponential map $\exp^{-1}_{\gamma(t)}(x(t))$; (ii) apply the unrolling map in Definition \ref{defs}(a) to the integral curve of this vector field, and translate to a new origin in $T_bM$. Consequently, the unwrapping map reduces to the unrolling map when $x=\gamma$ in \eqref{eqn:translated:unwrapping}; in other words, the unwrapping of a curve with respect to itself is its unrolling. 
\end{remark}

\begin{remark}
\label{remark:inverse:relation:of:unwrapping:wrapping}
Provided $x$ in (c) is outside the cut loci of $\gamma$, the wrapping map is the unique inverse of the unwrapping map. Provided $y_b$ in (d) is such that 
the argument of $\exp_{\gamma(t)}$ in 
\eqref{eqn:translated:wrapping}
is outside the tangential cut locus $\mathcal{T}\{\gamma(t)\}$ for all $t$, then the unwrapping map is the unique inverse of the wrapping map. For manifolds with empty cut loci, unwrapping and wrapping are hence always unique inverses. An example of such a manifold is the set of symmetric positive definite matrices considered in an example in \S\ref{example:rolled:model:on:SPD}.
\end{remark}

\begin{remark}
The maps in Definition \ref{defs}
either flatten or unflatten curves, and the superscript arrow notation introduced in the definitions is a convenient shorthand notation to make the distinction.
\begin{enumerate}
    \item [(i)] The unrolling and unwrapping maps are \emph{flattening} maps,  indicated by having downarrow `$\downarrow$' within superscript. Downarrow by itself in superscript means unrolling, whereas downarrow accompanied by a curve, say $\gamma$, means ``unwrapping with respect to $\gamma$''. For example, $x_b^{\downarrow}$ is the unrolling of curve $x$ into $T_bM$; and $x_b^{\downarrow \gamma}$ is the unwrapping of $x$ with respect to $\gamma$.
    Since unwrapping a curve with respect to itself is its unrolling, thus $x_b^{\downarrow x} = x_b^{\downarrow}$.
    \item[(ii)] The rolling and wrapping maps are \emph{unflattening} maps, indicated by uparrow `$\uparrow$' within superscript. Uparrow by itself in superscript means rolling,
    whereas uparrow accompanied by a curve, say $\gamma$, means ``wrapping with respect to $\gamma$''. For example, for a curve $y_b$ in $T_bM$, the curve $y_b^\uparrow$ is its rolling onto $M$; and $y_b^{\uparrow \gamma}$ is the wrapping of $y_b$ with respect to $\gamma$. 
\end{enumerate}
\end{remark}

\section{Rolled Gaussian process model}
\label{sec:rolled:GP}
The maps in the preceding section help define a generative model for random curves on $M$ from a model for curves in $\mathbb R^d$. Let $z(t)$ be a process in $\mathbb R^d$ with $E \{z(t)\} = m(t)$, and $y_b(t) = U z(t)$ be a process in $T_bM$ with $E \{y_b(t)\}=m_b(t) = U m(t)$ for an orthonormal frame $U: \mathbb R^d \rightarrow T_bM$. Let the ``rolled mean'' $\gamma = m_b^\uparrow$ be the  rolling of $m_b$ onto $M$, and $x = y_b^{\uparrow \gamma}$ be the wrapping of $y_b$ with respect to $\gamma$. In terms of notation, for a fixed $b\in M$: 
\begin{itemize}[leftmargin=*]
\itemsep 0em
    \item A subscript $b$ without a superscript distinguishes fixed or random curves and parameters on $T_bM$; 
    \item $z$ denotes an $\mathbb R^d$-valued stochastic process, and $x$ an $M$-valued process.
\end{itemize}

The following result elucidates on conditions under which the rolled mean $m_b^\uparrow$ equals the population Fr\'echet mean $\fre$ of $x$.

\begin{theorem}[Rolled mean coinciding with the Fr\'echet mean]
\label{thm:frechet:mean}
Assume that for every $t \in [0,1]$, the Fr\'echet mean of $x(t)$ exists and is unique. 
Then, for every $t$, the rolled mean, $\gamma(t)$, is the Fr\'echet mean of $x(t)$ if any of the following conditions are true:
\begin{enumerate}
\itemsep 0em
    \item [(i)] $M$ is a Hadamard manifold;
    \item[(ii)] $M$ is a symmetric space, the distribution of $x(t)$ has geodesic symmetry about $\gamma(t)$, and has a unique Fr\'echet mean that lies outside the cut locus of $\gamma(t)$.
\end{enumerate}
\end{theorem}
Theorem \ref{thm:frechet:mean} holds for general processes $z$ in $\mathbb R^d$ with finite mean, not necessarily Gaussian. It plays an important role in \S\ref{sec:asymp} where estimators of the mean function $m$ and covariance $k$ of a rolled Gaussian process are proposed and their asymptotic properties are studied. Manifolds in (i) of Theorem \ref{thm:frechet:mean} include the $d$-dimensional hyperbolic space $\mathbb H^d$ with the hyperbolic metric and the space $\text{Sym}_{>0}(d)$ of $d \times d$ real symmetric positive definite matrices with several choices of Riemannian metrics; examples of symmetric manifolds relevant to condition (ii) are the non-compact $\mathbb H^d$ and $\text{Sym}_{>0}(d)$, and compact manifolds such as $\mathbb S^2$, rotation matrices, and Lie groups with bi-invariant metrics. 


%

We now define the rolled Gaussian process model. Let $z \sim GP(m,K)$ be a time-indexed Gaussian process in $\mathbb R^d$ with mean curve $t \mapsto m(t)=\{m^1(t),\ldots,m^d(t)\}^\top \in \mathbb R^d$ and covariance $(t,t') \mapsto K(t,t') \in \text{Sym}_{>0}(d)$. The law of $z$ is characterised by the requirement that for every finite $r$ and times $t_1,\ldots,t_r$, the $d \times r$ matrix $Z:=\{z(t_1),\ldots,z(t_r)\}$ satisfies $\operatorname{vec} Z \sim N_{dr}(\vc M, \Sigma)$, where $M:=\{m(t_1),\ldots,m(t_r)\} \in \mathbb R^{d \times r}$, and $\Sigma$ is block diagonal with $K(t_i,t_j) \in \text{Sym}_{>0}(d)$ in the $(i,j)$th block. 

Given $b \in M$ and a curve $\gamma$ on $M$, let $U=\{u_1,\ldots,u_d\}$ be an orthonormal basis on $T_b M$ in its matrix representation. With 
$m_b(t)=U m(t):=\sum_{i=1}^d m^i(t)u_i$ and $K_b(t,t')=U K(t,t')U^\top$, consider the process $y_b(t):=U z(t)$. On the inner product space $(T_bM,\langle \cdot, \cdot \rangle_{b})$ note that $y_b(t)$ is a Gaussian random vector with mean $m_b(t)$ and covariance determined by the nonnegative linear operator $K_b(t,t):T_{b}M \to  T_{b}M$; see Appendix \ref{sec:Gaussian} for details on Gaussians on inner product spaces. 
As a consequence, $y_b(t):=U z(t) \sim \mathcal{GP}(m_b, K_b)$ assuming values in $T_{b}M$. We can thus identify a Gaussian process in $T_{b} M$ with a Gaussian process $z \sim \mathcal{GP}(m,K)$ in $\mathbb R^d$, that when transformed using the wrapping map in Definition \ref{defs}(b) with respect to a curve $\gamma$ in $M$, results in a random curve on $M$ whose law is determined by the pair $(b,U)$ and parameters $(m, K)$.\\

\begin{definition}[Rolled Gaussian process]
 Let $z \sim GP(m,K)$, and choose $b \in M$ and a frame $U:\mathbb R^d \to T_bM$. 
 With $m_b=Um$ as a curve in $T_bM$, let $\gamma:= m_b^\uparrow$ be the rolling of $m_b$, and let $x=y_b^{\uparrow \gamma}$ be the wrapping of $y_b:=Uz$ with respect to $\gamma$. The random curve $x$ on $M$ is a rolled Gaussian process, denoted $x \sim \mathcal{RGP}(m, K; b, U)$.
\end{definition}

Of particular note in the definition above is the fact that the Gaussian process $y_b$ in $T_bM$ is \emph{wrapped with respect to the rolling of its mean} function $m_b$. This is a convenient modelling feature because the curve $\gamma$ then is not an additional parameter to be estimated, but is  instead fully determined by the mean of $z$. Figure \ref{fig:Rd:TM:M:schematic} illustrates the construction of the rolled Gaussian process and notation.

\begin{figure} [!t]
    \centering
  \includegraphics[scale = 0.35]{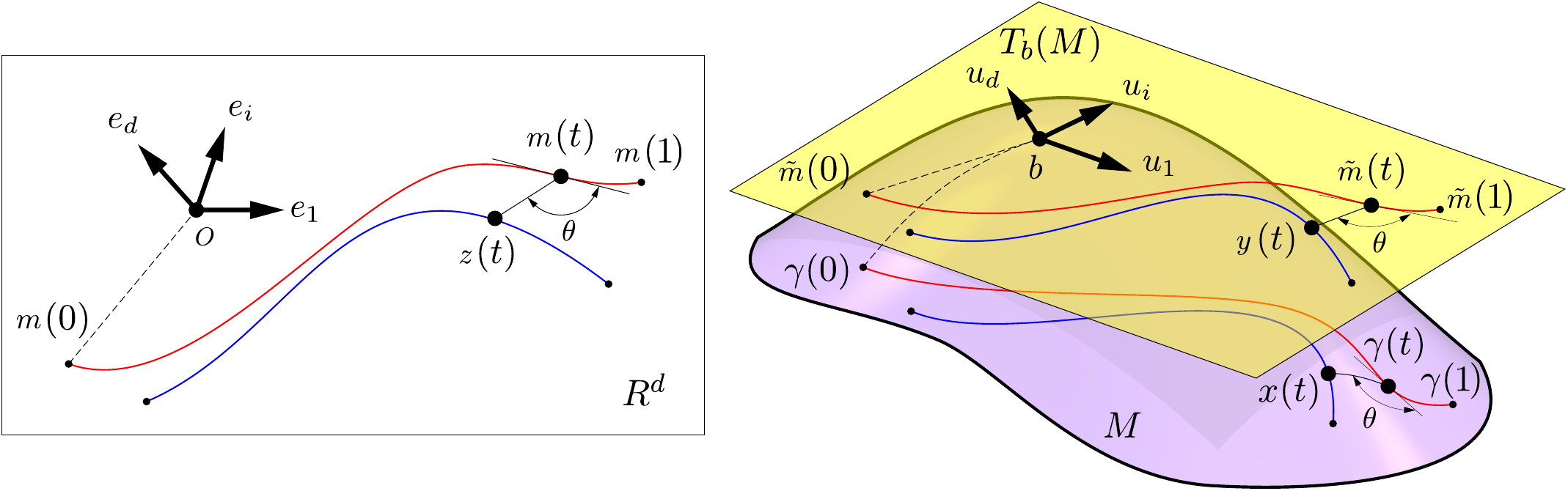}
  \caption{Illustration of how curves defined in $\mathbb R^d$ are identified on $M$, via an intermediate tangent space $T_bM$. Red is the mean curve with respect to which the (un)flattening is performed, and blue is a different curve that deviates from the mean curve. The line segment connecting points  between the blue and red curves at arbitrary $t$, 
  and the corresponding angle denoted $\theta$, indicate distances and angles preserved by the (un)flattening.}
\label{fig:Rd:TM:M:schematic}
\end{figure}

The rolled Gaussian process involves a random vector field along $\gamma$ arising from the wrapping operation. Proposition \ref{prop:vector_field} below identifies it as a Gaussian vector field, which we claim is a natural one. To understand why, we note that if Gaussian random variables $v(p) \sim N_d(\mu(p), \Sigma(p))$ on $T_pM$ and $v(q) \sim N_d(\mu(q), \Sigma(q))$ on $T_qM$ are such that
\[
\mu(q)=P^\gamma_{0 \to 1}\mu(p), \quad \Sigma(q)=P^\gamma_{0 \to 1} \Sigma(p)P^\gamma_{0 \leftarrow 1},
\]
where $\gamma:[0,1] \to M$ is a piecewise smooth curve taking $\gamma(0)=p$ to $\gamma(1)=q$. Then, $v(q)\overset{d}=P^\gamma_{0 \to 1}v(p)$ and it is easy to see using geodesic normal coordinates that $v(q)$ and $v(p)$ are Gaussian random vectors on two isometric copies of $\mathbb R^d$. Back to our setting, consider the random field $t \mapsto V^\gamma(t) \in T_{\gamma(t)}$, $V^\gamma(t):=P^\gamma_{t \leftarrow0}P^c_{1\leftarrow 0}(y_b(t)-m_b(t))$. The tangent vectors $V^\gamma(s) \in T_{\gamma(s)}M$ and $V^\gamma(t) \in T_{\gamma(t)}M$ are related via parallel transport maps, and each a zero-mean Gaussian since parallel transports are linear isometries. Thus $V^\gamma$ along $\gamma$ may be viewed as an isometric copy of the $\mathbb R^d$-valued Gaussian process $z$. Adapting terminology from the deterministic setting, we say that the vector field $V^\gamma$ is \emph{covariantly Gaussian}, and in this sense, the rolled Gaussian process is the natural curved analogue of the Euclidean Gaussian process.
\begin{proposition}[Gaussian vector field along $\gamma$]
\label{prop:vector_field}
Given $(b,U)$, $\{V^\gamma(t), t \in [0,1]\}$ is a covariant Gaussian vector field along $\gamma$. 
\end{proposition}

The choice of point $b \in M$ and basis $U$ on its tangent space $T_bM$ is arbitrary. The rolled Gaussian process, however, is equivariant with respect to this choice so that in practice the particular choice of $b$ and $U$ is inconsequential.
\begin{proposition}[Equivariance of rolled Gaussian process]
\label{prop:equivariance}
Starting from point $b \in M$ with basis $U$ for  $T_bM$, let $x \sim \mathcal{RGP}(m, K; b, U)$. If one starts instead from $b' \in M$ with basis $U'$ for $T_{b'}M$, then there exists unique $m'$ and $K'$ such that the rolled Gaussian process $ x' \sim \mathcal{RGP}(m', K'; b', U')$ is equal in distribution to $x$.
\end{proposition}

\subsection{Parametric models}
\label{sec:parametric:model:for:curves:on:Rd}
A core ingredient of the rolled Gaussian process model  is a Gaussian process model, $z \sim \mathcal{GP}(m,K)$, for curves in $\mathbb R^d$ to be rolled and wrapped onto $M$. This model requires specification of the mean, $m$, and covariance, $K$, functions. For many applications it is beneficial to consider $z$ in a basis representation, which may be defined as follows. Let $\{\phi_s: [0,1] \rightarrow \mathbb R\}$ be a linearly independent basis for real functions on $[0,1]$, which includes the constant function. Let $z(t)=\sum_{s=1}^\infty \xi_s\phi_s(t)$, where $\xi=(\xi_1,\xi_2,\ldots)$ is a random Gaussian element in a Hilbert space $H$. Its mean $\mathbb Ez(t)=m(t)=K_0\Phi(t)$, where $\Phi(t)=(\phi_1(t),\phi_2(t),\ldots) \in \ell^2$, and $K_0: \ell^2 \to \mathbb R^d$ is a linear map; its covariance $\Xi:H \times H$ given by the tensor product $\Xi=V \otimes U$, where $U \in \text{Sym}_{>0}(d)$ and $V: \ell^2 \to \ell^2$ is a bounded, self-adjoint operator on the space $\ell^2$ of square-summable real sequences. Thus the associated covariance function $K(s,t)=\langle\Phi(s),V\Phi(t) \rangle_H$U. 

The special case with a finite number, $k$, of basis functions is particularly useful when data curves are observed at discrete times, as in the robot application in \S\ref{sec:robot:SO3}. 
For example, let $\{\phi_s: [0,1] \rightarrow \mathbb R\}$ be a B-spline basis \citep{deBoor2002practical, ramsay2009fda} of order 4 and knots at the $k-2$ interior points of a grid of $k$ equally spaced points on $[0,1]$ with $k\geq \operatorname{max}(4,d)$; then each $\phi_s$ is a piecewise cubic polynomial with joins at the knots such that any adjacent polynomials' values, and values of first and second derivatives, are matched \citep{deBoor2002practical}. By construction, $\sum_{s=1}^k \phi_s(t) = 1$, for any $t \in [0,1]$; that is, the constant function is in the span of this basis, a property useful later in Proposition \ref{prop:equivariance:in:terms:of:basis}. The Gaussian process $z$ can be written as
\begin{equation}
z(t) = \sum_{s=1}^k {w_s} \phi_s(t),
\label{eqn:y:basis:expansion}
\end{equation}
where
$W = (w_1, \ldots, w_k) \sim \mathcal {MN}(M_w, U_w, V_w)$, the matrix normal distribution with mean matrix, $M_w \in \mathbb R^{d \times k}$, row covariance, $U_w \in \SPD(d) $, and column covariance $V_w \in \SPD(k)$. Then, its mean $m(t) = M_w \phi(t)$, and separable covariance $K(t, t') = \phi(t)^\top V_w \phi(t') U_w$, where $\phi(t) = \{\phi_1(t), \ldots, \phi_k(t) \} \in \mathbb R^k$ and $\SPD(p)$ is the cone of $p \times p$ symmetric positive definite matrices.

\section{Estimation for fully observed curves}
\label{sec:asymp}
In this section, we develop estimation theory for the mean and covariance $m_b$ and $K_b$ of the 
$\mathbb R^d$-valued Gaussian process $y_b=Uz$, when the rolling $m^\uparrow$ of $m_b$ onto $M$ equals the Fr\'echet mean $\fre$ of the rolled Gaussian process  $x \sim \mathcal{RGP}(m,K;b, U)$. Theorem \ref{thm:frechet:mean} identifies at least two types of $M$ which are of interest in statistics for which $m_b^\uparrow=\fre$. Estimators of $m$ and $K$ are then obtained by transforming the estimators of $m_b$ and $K_b$ under $U^\top: T_bM \to \mathbb R^d$. 

We first consider the case when $x_1,\ldots,x_n$ are independent realizations of the process $x \sim \mathcal{RGP}(m,K;b, U)$, where each $x_i$ is densely observed on $[0,1]$ so that a smoothed or interpolated estimate of each may be considered as continuously observed. The unrolling and unwrapping procedures transfer entire curves onto a single tangent space, which upon choosing a frame is identified with $\mathbb R^d$. 
    
If $\mathbb P$ is the law of $x$, the marginal laws of $x(t)$ for each $t$ are each absolutely continuous with respect to the volume measure on $M$ since they are each a pushforward measure, under the surjective exponential map, of a non-degenerate Gaussian  defined on a tangent space. Therefore, $\mathbb P(\mathcal C(x([0,1]))=0$ since the cut locus of any point is of dimension $d-1$. Moreover, since $p \in \mathcal C(q)$ if and only if $q \in \mathcal C(p)$, the random variable $\exp^{-1}_p(x(t))$ is well-defined almost surely, for every $t$ and any $p \in M$. Recall the population and sample Fr\'echet mean curves in  \ref{eq:Fre_population} and \ref{eq:Fre_mean_curve}.
%
We assume that
\begin{description}
    \item[A1.] $\mathbb E[\rho^2(p,x(t))]<\infty$ for all $p \in M$ and $t \in [0,1]$, and the Fr\'echet mean $\gamma_{\text{Fre}}(t)$ exists and is unique for every $t \in [0,1]$. Further, assume that there are at least two minimal geodesics from $\gamma_{\text{Fre}}(t)$ to any point in its cut locus. 
\end{description}
To assume uniqueness of the sample and population Fr\'echet means is standard \citep[e.g.][]{dai2018principal, lin2019intrinsic}, without which it is difficult to probe the asymptotic properties of estimators of the mean $m$ and covariance $K$ of the Gaussian process $z$; see, for example, \cite{afsari2011riemannian} for conditions that ensure this pointwise for each $t$. The assumption of existence of two minimal geodesics implies that $\mathbb P(\mathcal C(\fre([0,1]))=0)$ \citep{le2014measure}, and enables us to avoid assuming the more restrictive condition that $x(t)$ lies outside the cut locus of $\gamma_{\text{Fre}}(t)$. A recent result suggests that the assumption on the existence of at least two minimal geodesics may be dispensed altogether \citep{lytchak2025zero}.




Completeness of $M$ guarantees that $\hatfre(t)$ is well-defined and almost surely unique \citep[Corollary 2.4][]{arnaudon2014means}, and necessarily a zero of the derivative of $F_n$. We prove in Theorem \ref{thm:C1_convergence} that under some assumptions that follow, $\hatfre$ with high probability lies within a tubular neighbourhood around $\fre$ defined later in \eqref{eq:tubular}. 

Context for the remaining assumptions is provided by the choice of our estimators. 
Under the assumption that the rolling $m_b^\uparrow$ of mean $m_b$ of the  Gaussian process $z$ coincides with $\fre$, a natural estimator of $m_b$ is 
\begin{equation}
\label{eq:rolled_mhat}
\hat m_b:=(\hatfre)^{\downarrow}_b,
\end{equation}
the unrolling of the sample Fr\'echet mean curve $\hatfre$ onto $T_bM$. The estimator $\hat m_b$ equals the estimator
\begin{equation}
\label{eq:m_2}
\frac{1}{n}\sum_{i=1}^n x^{\downarrow \hat \gamma_{\text{Fre}}}_{b,i}(t),
\end{equation}
where
\begin{equation}
x^{\downarrow \hat \gamma_{\text{Fre}}}_{b,i}(t):=\hat \gamma_{\text{Fre}}^\downarrow(t)+P^{\hat c}_{0 \leftarrow 1}P^{\hat \gamma_{\text{Fre}}}_{0 \leftarrow t} \exp^{-1}_{\hat \gamma_{\text{Fre}}(t)} \left\{ x_i(t) \right\}, \quad t \in [0,1],
\label{eq:xib}
\end{equation}
is the unwrapping of $x_i$ onto $T_bM$ with respect to the sample Fr\'echet mean curve $\hat \gamma_\text{Fre}$,
where $\hat c$ estimates the geodesic $c$ from $\gamma_{\text{Fre}}(0)$ to $b$ in $M$. This follows from the fact that parallel transports are isometries and preserve the origin, and under assumption A1 $\hatfre(t)$ is characterised by the relation
\[
\frac{1}{n}\sum_{i=1}^n \exp^{-1}_{\hat \gamma_{\text{Fre}}(t)} \{x_i(t)\}=0.
\]
For the estimator of covariance $K_b$ we consider
\begin{equation}
\label{eq:cov_est}
\hat K_b(s,t):=\frac{1}{n}\sum_{i=1}^n (x^{\downarrow \hat \gamma_{\text{Fre}}}_{b,i}(s)-\hat m_b(s))(x^{\downarrow \hat \gamma_{\text{Fre}}}_{b,i}(t)-\hat m_b(t))^\top.
\end{equation}
Computation of the estimators rely on first computing the sample Fr\'echet mean curve $\hat \gamma_{\text{Fre}}$, from which the $x^{\downarrow \hat \gamma_{\text{Fre}}}_{b,i}$ are computed. The convergence, and the subsequent rates, 
of $\hat m_b$ and $\hat K_b$ thus depend on three sources of error: (i) using $\hat \gamma_{\text{Fre}}$ in place of $\gamma_{\text{Fre}}$; (ii) using the vector field $\exp^{-1}_{\hat \gamma_{\text{Fre}}(t)} \left\{ x(t) \right\}$ in place of $\exp^{-1}_{\gamma_{\text{Fre}}(t)} \left\{ x(t) \right\}$; (iii) using parallel transport maps along $\hat \gamma_{\text{Fre}}$ and $\hat c$ instead of $\gamma_{\text{Fre}}$ and $c$. 

For the error in (i) the uniform rate 
 \begin{equation}
\label{eq:uniform_convergence}
\sup_{t \in [0,1]}\rho(\hat \gamma_{\text{Fre}}(t),\gamma_{\text{Fre}}(t))=O_P(n^{-1/2})
 \end{equation}
has been obtained under certain \emph{global} conditions on $M$ and the law of $x$. For example, \cite[Assumption B3]{dai2018principal} assume that the sample paths of $x$ almost surely cluster around $\gamma_{\text{Fre}}$, while \cite[Assumption B3]{lin2019intrinsic} impose a strong global moment condition on $x$ for \emph{all} compact subsets of $M$. However, in our setting exponential decay of the tail probabilities of the Gaussian process $z$ can be exploited to control the fluctuations of the sample paths of $x$ within a neighbourhood of $\fre$. Assume that
\begin{description}
    \item[A2.] $M$ has bounded geometry, i.e. $|\sec{(M)}| \leq \kappa$ for $\kappa \geq 0$. 
\end{description}
Consider the closed tubular neighbourhood of the Fr\'echet mean curve $\fre$,
\begin{equation}
\label{eq:tubular}
\mathbb T_{\gamma_{\text{Fre}}}(r_0):=\{\exp_{\gamma_{\text{fre}}(t)}(v): t \in [0,1], \|v\|_{\gamma_{\text{fre}}(t)} \leq r_0\},    
\end{equation}
where $r_0<\min\left\{\frac{\pi}{2\sqrt{\kappa}},\frac{1}{2} \text{inj}(\gamma_{\text{Fre}}([0,1]))\right\}$. Lemma \ref{lem:tubular} in Appendix shows that $\mathbb T_{\gamma_{\text{Fre}}}(r_0)$ is well-defined. Under assumption A2, when assumption A4 stated below also holds, all sample paths of $x$ lie in $T_{\gamma_{\text{Fre}}}(r_0)$ with high probability, which suffices to prove \eqref{eq:uniform_convergence}. 

However, we make a stronger set of consolidated assumptions on the paths of $x$, since the study of errors (ii) and (iii) require convergence of vector fields along $\hat \gamma_{\text{Fre}}$, for which we require that $\hat \gamma_{\text{Fre}}$ converges in probability to $\gamma_{\text{Fre}}$ in the weak $C^1$ topology (see Appendix \ref{app:c_k curves} for a precise definition). $C^1$ convergence of $\hat \gamma_{\text{Fre}}$ has not been considered before in the literature, and is one of our key theoretical contributions. 

\begin{description}
\itemsep 0em
    \item [A3.] The sample paths of the Gaussian process $z$ are in $C^2([0,1], \mathbb R^d)$ almost surely.
    \item[A4.] The mean $m$ and covariance $K$ of the $\mathbb R^d$-valued Gaussian process $z$ satisfy 
  \[
    \zeta:=\sup_{t \in [0,1]}\|m(t)\| <r_0/2, \quad \sigma^2:=\sup_{s,t \in [0,1]^2} \|K(s,t)\|_{\text{op}}<\infty. 
  \]
    \item [A5.]  $\inf_{t \in [0,1]} \lambda_{\min}(\mathbb EH_t)>0$, 
    where $p \mapsto H_t(p)$ is the Hessian operator corresponding to the Hessian of $p \mapsto \rho^2(p,x(t))$, and  $\lambda_{\min}$ is its smallest eigenvalue. 
    
\end{description}
Assumption A2 is satisfied by many $M$ that are of interest in statistics; these include compact manifolds such as $\mathbb S^d, O(d), SO(d)$; the Stiefel and Grassmannian manifolds with standard choices of metrics; non-compact manifolds such as $\mathbb H^d$, and the manifold $\text{Sym}_+(d)$ of $d \times d$ symmetric positive definite matrices with several Riemannian metrics (e.g. log-Euclidean; affine-invariant; power-Euclidean; Cholesky; Bures-Wasserstein). It is needed to control behaviour of the Hessian of $p \mapsto \rho^2(p,q)$ whenever $q$ is close to the cut locus of $p$, on which the Hessian is unbounded, and plays an important role when deriving bounds on empirical processes indexed by the class of functions $\{\nabla^2_p \rho^2(p,\cdot)\}$ used in the proof of Theorem \ref{thm:C1_convergence}.

Assumption A3 is not very restrictive. For example, it is satisfied by a $z$ with separable covariance $K(s,t)=Pk(s,t)$ for a scalar Matern kernel $k:[0,1]^2 \to \mathbb R$ with smoothness parameter $5/2$ and a positive definite matrix $P$. Assumption A2 implies that the sample paths of the rolled Gaussian process $x$ are in $C^2([0,1], M)$ almost surely; see Appendix \ref{app:c_k curves} for a precise definition of $C^2$ curves on $M$ and the $C^2$ topology. To see this, with $\gamma=m^\uparrow$, note that the curve $(\gamma(t),V^\gamma(t))$ is in $C^2([0,1],TM)$, where $V^\gamma$ is the Gaussian vector field along $\gamma$ in Proposition \ref{prop:vector_field}. The claim follows since the exponential map $\exp:TM \to M$ is smooth on a complete manifold $M$. Mote that $C^2$ regularity of the paths of $x$ is needed to ensure that $\hat \gamma_{\text{Fre}}$ and $\gamma_{\text{Fre}}$ are in $C^2([0,1],M)$ (Theorem \ref{thm:C1_convergence}), and that the unrolling or unwrapping of a curve with respect to $\hat \gamma_{\text{Fre}}$ does not deviate too much from the corresponding actions with respect to $\gamma_{\text{Fre}}$ (Theorem \ref{thm:rates}). 

Assumption A4 enables exponential decay of probability of sample paths outside of the tubular neighbourhood $\mathbb T_{\gamma_{\text{Fre}}}(r_0)$ in \eqref{eq:tubular}, and is weaker than requiring all sample paths of $x$ to lie within a subset of $M$. We thus prove in Lemma \ref{lem:small_prob} that the distance between a sample path of $x$ and sample Fr\'echet mean $\hatfre$ is bounded in probability. The condition $\zeta < r_0/2$ is mildly restrictive: for $M$ that is a symmetric space or a manifold with non-positive curvature Theorem \ref{thm:frechet:mean} shows that the rolling of $m_b$ on to $M$ coincides with $\fre$, and the condition is trivially satisfied. For general $M$, the condition requires the rolling of $m_b$ to not deviate too far from $\fre$. 


Assumption A5 is needed in the proof of Theorem \ref{thm:C1_convergence} for similar reasons as with assumption A2 when controlling the Hessian of the squared-distance function. It first appeared in \cite{dai2018principal} and later in \cite{lin2019intrinsic, shao2022intrinsic}. It is a strong condition, satisfied, for example, on manifolds with no conjugate points (e.g. non-positively curved $M$ such as  $\mathbb H^d$ and $\text{Sym}_+(d)$). It generally fails to hold on positively curved manifolds $M$ except under restriction on the support of $x$. Given assumption A1, we conjecture that Theorem \ref{thm:C1_convergence} holds even without assumption A5, but it is unclear to us how to prove it. 

Under assumptions A1-A5, Theorem \ref{thm:C1_convergence} below clarifies the smoothness of $\fre$ and $\hatfre$, and establishes $C^1$ convergence of $\hatfre$ to $\fre$ at the optimal Euclidean rate. As far as we are aware, this result in novel and has not appeared before in the literature. 
\begin{theorem}
\label{thm:C1_convergence}
Under assumptions A1-A5, 
\begin{enumerate}
\itemsep 0em
    \item [(a)] $\gamma_{\text{Fre}}$ belongs to $C^2([0,1], M)$;
    \item[(b)] The sample Fr\'echet mean $\hatfre$ lies in $\mathbb \mathbb \mathbb T_{\gamma_{\text{Fre}}}(r_0)$ and belongs to $C^2([0,1], M)$ a.s.;
    \item[(c)] $\hat \gamma_{\text{Fre}}$ converges to $\gamma_{\text{Fre}}$ as $n \to \infty$ in the $C^1$ topology at the rate $O_P(n^{-1/2})$. 
\end{enumerate}
\end{theorem}
Errors of estimators of $m_b$ and $K_b$ are largely governed by the error $\|x^{\downarrow\hat \gamma_{\text{Fre}}}_{b,i}-x^{\downarrow
\gamma_{\text{Fre}}}_{b,i}\|_b$ that is induced by using $\hatfre$ in place of $\fre$ for the unwrapping. The error may be compared with the situation in the Euclidean case, where $\|(x_i-\bar x)-(x_i-\mu)\|=O_P(n^{-1/2})$ with $\bar x$ and $\mu$ as the pointwise sample and population means, respectively;  a similar rate is true for the deviation of the sample covariance from its population counterpart. Theorem \ref{thm:rates} below shows that this remains true even for $C^2$ curves on a manifold $M$: the rates of convergence of $\hat m_b$ and $\hat K_b$ are determined by the rate of convergence of $\hat \gamma_{\text{Fre}}$ to $\gamma_{\text{Fre}}$ in Theorem \ref{thm:C1_convergence}, and match the optimal rates in the Euclidean setting. Recall that $\|\cdot\|$ is the norm and $\|\cdot\|_F$ is the corresponding Frobenius norm on $\mathbb R^d$. 
\begin{theorem}
 \label{thm:rates}
Assume A1-A5. For a frame $U:\mathbb R^d \to T_bM$, the following hold:
 \begin{enumerate}
\item [(a)] $\underset {t \in [0,1]} \sup \|U^\top \hat m_b(t)-U^\top m_b(t)\|=
\underset {t \in [0,1]} \sup\|\hat m_b(t)-m_b(t)\|=
O_P(n^{-1/2})$;
     \item[(b)] $\underset {(s,t) \in [0,1]^2} \sup \|U^\top (\hat K_b(s,t)- K_b(s,t)) U\|_F=\underset {(s,t) \in [0,1]^2} \sup\|\hat K_b(s,t)- K_b(s,t)\|_F=O_P(n^{-1/2})$.
 \end{enumerate}
\end{theorem}
\begin{remark}
\label{rem:param_rates}
In the parametric model under a basis respresentation of the Gaussian process $z$ in \S\ref{sec:parametric:model:for:curves:on:Rd} with finite $k$ number of basis functions, the estimators (MLEs) of the mean matrix, $M_w$, and row and column covariances, $U_w$ and $V_w$, can be defined in two steps: given $U:\mathbb R^d \to T_bM$, project $x^{\downarrow\hat \gamma_{\text{Fre}}}_{b,i}$ onto the $k$-dimensional subspace spanned by the basis functions to obtain $\hat w_{i,s}, s=1,\ldots,k$; then for each $s$,  the projection of $1/n\sum_{i=1}^n\hat w_{i,s}$ onto the $k$-dimensional subspace is the estimator of column $s$ of $UM_w$, and estimators of the covariances $U_w$ and $V_w$ are computed from the $\hat w_{i,s}$, using an iterative algorithm since no-closed form is available, under an identifiability constraint (e.g. $\text{Tr} (U_w)=d$ or $\text{det}(U_w)=1$)  \citep{dutilleul1999mle}. 
More details are provided in \S\ref{sec:estimation} when estimators for discretely observed curves are discussed. These estimators are will have identical convergence rates to $\hat m$ and $\hat K$ in Theorem \ref{thm:rates}. 

When $k \to \infty$, convergence rate of the estimator of $U_w \in \text{Sym}_{>0}(d)$ remains unchanged. If $\hat M_w$ and $\hat V_w$ are the estimators of $M_w$ and $V_w$, it can be shown that
\[
\|\hat M_w-M_w\|_F=O_P\left(\sqrt{\frac{dk}{n}}\right),\quad \|\hat V_w-V_w\|_F=O_P\left(\frac{k}{\sqrt{n}}\right),
\]
and convergence for $\hat M_w$ requires $k=o(n)$, and convergence for $\hat V_w$ requires $k=o(\sqrt{n})$. 
\end{remark}
We conclude this section with an example of a rolled Gaussian process $x$ with full support on $C^2([0,1],M)$ that satisfies assumptions A1-A5, and hence for which the asymptotic theory in Theorems \ref{thm:C1_convergence} and \ref{thm:rates} applies without caveats. 
\begin{example}
Let $M$ be $\text{Sym}_{>0}(q)$ with $d=q(q+1)/2$ equipped with the affine invariant Riemannian metric; see Appendix \ref{sec:appdx:expressions:for:SPD} for details. Since $M$ is a Hadamard manifold, the population Fr\'echet mean curve $\fre$ for a rolled Gaussian process obtained from an $\mathbb R^d$-valued Gaussian process exists and is unique, and assumption A1 is thus satisfied.
Assumption A5 is satisfied since the second moment of $z$ is finite and $M$ has no conjugate points. Moreover, $-1/2 \leq \sec(p) \leq 0$ for every $p \in M$ \citep[e.g.][Chapter 4]{thanwerdas2022riemannian}, and assumption A4 is satisfied with $\kappa=1/2$. 

Consider the parametric model in \S\ref{sec:parametric:model:for:curves:on:Rd}, where the $\mathbb R^d$-valued Gaussian process $z(t)=\sum_{s=1}^\infty \xi_s \phi_s(t)$ with $\{\phi_s(t):[0,1] \to \mathbb R\}$ as the spline basis. Then, the mean $\mathbb Ez(t)=m(t)=R_0\Phi(t)$ with $\Phi(t)=(\phi_1(t),\phi_2(t),\ldots)$ and the map $R_0: \ell^2 \to \mathbb R^d$ with $R_0(w):=(\langle a_1,w \rangle_{\ell^2},\cdots,\langle a_d,w \rangle_{\ell^2})^\top \in \mathbb R^d$ for sequences $a_i \in \ell^2, i=1,\ldots,d$. Note also that by Theorem \ref{thm:frechet:mean} the rolling of the mean function $m$ equals $\fre$, since $M$ is a Hadamard manifold. Then, $R_0$ is bounded with 
\[
\|R_0(w)\|_{\text{op}} \leq \sqrt{\sum_{i=1}^d \|a_i\|^2_{\ell^2}}.
\]
For the spline basis, let $\Phi_s(t):=(\phi_1(t),\ldots,\phi_s(t))$, with $\sup_{s,t}\|\Phi_s(t)\|<\infty$, since at any given $t$ only a fixed number of basis functions, depending on the order of the spline independent of $s$, are non-zero with each being uniformly bounded. Thus, the mean $m$ is uniformly bounded. The covariance of $z$ equals $V \otimes U$, where $V:\ell^2 \to \ell^2$ of $z$ is bounded and self-adjoint and $U \in \text{Sym}_{>0}(d)$. It is bounded if $\|V\|_{\text{op}} <\infty$, where the operator norm is defined by $\|V\|_{\text{op}}:=\sup\{\|Vw\|_{\ell^2}: w \in \ell^2, \|w\|_{\ell^2}\}$. For example, the diagonal operator $V=\text{diag}(\lambda_1,\lambda_2,\dots)$ with $\lambda_i \geq 0$ and $\sup_i \lambda_i<\infty$ is bounded (e.g. $\lambda_i=(1+i)^{-2}$). Thus assumption A3 is satisfied. Finally, smoothness of the spline basis and assumption A3 ensure that the sample paths of the rolled Gaussian process satisfy assumption A2. 

\end{example}

\section{Estimation and testing for discretely observed curves under a parametric model}
In practice, curve data are recorded in discrete time, and a discrete version of the rolled Gaussian process model is accordingly needed. In this section, we use the parametric model in \S\ref{sec:parametric:model:for:curves:on:Rd} with a finite number, $k$, of basis functions to carry out inference in the discrete-time setting. 
In the discrete setting, measurement noise and the observation times may affect computation, and we thus discuss several practical estimators that may be used even when the rolling of the mean $m_b(t)=M_w\phi(t)$ does not equal $\fre$. The estimator described in Remark \ref{rem:param_rates} related to the discretized version of the estimator $\hat m_b$ in \eqref{eq:rolled_mhat}, as remarked, is indeed consistent. We do not consider the setting where curves are observed at discrete times with additional measurement noise. This too is a practically important situation, and will be taken up in future work. 

Let $t_1, \ldots, t_r$ henceforth be times at which each curve is recorded and stored, with times $t_i = (i-1)/(r-1); i = 1,\ldots,r$ assumed equally spaced on $[0,1]$, and with $r > k$.
There is some loss of generality in assuming curves are recorded at common times, and that the observation times are equally spaced, but the former simplifies notation considerably, and both can be easily relaxed. We use upper case for the discrete-time version of the curve denoted in corresponding lower case; for example, $\Gamma := \{\gamma(t_1), \ldots ,\gamma(t_r)\} \in M^r$ is curve $\gamma$ in discrete time.

The four maps in \S\ref{sec:four:maps} have natural analogues for discrete curves: the unrolling, 
$\Gamma_b^\downarrow$, of $\Gamma$, is the discrete-time curve $\Gamma_b^\downarrow := \{\gamma_b^\downarrow(t_1), \ldots, \gamma_b^\downarrow(t_r) \} \in (T_b M)^r$, where $\gamma_b^\downarrow$ is the unrolling into $T_bM$ of the piecewise geodesic curve, say $\tilde{\gamma}$, that has geodesic segments connecting consecutive points $\gamma(t_{i})$ and $\gamma(t_{i+1})$, such that $U^\top\Gamma_b^\downarrow \in \mathbb R^{d \times r}$ for a frame $U:\mathbb R^d \to T_b M$.
Similarly, $(X)_b^{\downarrow {\Gamma}}$ with
\[
\gamma_b^\downarrow(t_j)+P^c_{0 \leftarrow 1}P^\gamma_{0 \leftarrow t_j} \exp^{-1}_{\gamma(t_j)} \left\{ x(t_j) \right\}
\]
as the $j$th component is the unwrapping of a discrete-time curve $X \in M^r$ with respect to the same piecewise geodesic $\tilde{\gamma}$ 
such that $U^\top (X)_b^{\downarrow {\Gamma}} \in \mathbb R^{d \times r}$. The unflattening maps, rolling and wrapping, for discrete curves follow similarly to these flattening ones; see Appendix \ref{sec:computations:for:rolling:etc:in:discretised:time}.

The discrete-time version of the parametric model in \S\ref{sec:parametric:model:for:curves:on:Rd} with a parametric specification of $m$ and $K$
can be written as 
\begin{equation}
Z \sim \mathcal{MN}\left(M_w \Phi, U_w, \Phi^\top V_w \Phi \right),
\label{eqn:Z:as:MN}
\end{equation}
where $\Phi = \left\{\phi(t_1), \ldots, \phi(t_r) \right\} \in \mathbb R^{k \times r}$ and $Z = \left\{z(t_1), \ldots, z(t_r) \right\} \in \mathbb R^{d \times r}$ are matrices. Then, the corresponding discrete-time version of model $x \sim \mathcal{RGP}(m, K; b, U)$ is denoted by
\begin{equation}
X \sim \mathcal{RMN}(M_w, U_w, V_w; b, U).
\label{eqn:discrete:time:model:for:curves:on:M}
\end{equation}
A random draw from \eqref{eqn:discrete:time:model:for:curves:on:M} is simulated by drawing random $Z$ from \eqref{eqn:Z:as:MN} then, with respect to base point $b$ and basis $U$, wrapping $Z$ with respect to the rolling of the mean $M = M_w \Phi$. We spell out these steps in \S\ref{sec:calcs:in:embedded:coordinates}
in practical embedding coordinates that are convenient for numerical computations.
Section \S\ref{example:rolled:model:on:S2} shows a concrete example in the special case with manifold $M = \mathbb S^2$, the unit 2-sphere in $\mathbb R^3$, of the model \eqref{eqn:discrete:time:model:for:curves:on:M} being prescribed and simulated from; the simulations are shown in  Fig.~\ref{fig:curves:from:toy:model:R2:S2}.

Prior to considering estimation in the next section, it is helpful first to define the following \emph{unwrapping coordinates} for the discretised curves. 

\begin{definition}[Unwrapping coordinates]
\label{def:unwrap_coordinates}
Given curves $\gamma$ and $x$ in $M$, consider, respectively, their discrete-time versions $\Gamma=\{\gamma(t_1),\ldots,{\gamma}(t_r)\}$ and
$X=\{x(t_1),\ldots,{x}(t_r)\}$. Given $(b,U)$, when $x$ and 
$\gamma$ are outside the cut loci of each other, and $b$ is outside the cut locus of $\gamma(0)$,  the unwrapping coordinates of $X$ are defined by the map
\begin{equation}
\label{eqn:unwrapping:coords:H}
(X,\Gamma) \mapsto H(X,\Gamma):=U^\top (X)_b^{\downarrow {\Gamma}}.
\end{equation}
\end{definition}
This definition reverses the steps described for defining $X$ in \eqref{eqn:discrete:time:model:for:curves:on:M}, in which the $X$ were defined in terms of $Z$. Indeed, if the conditions of Remark \ref{remark:inverse:relation:of:unwrapping:wrapping} hold, and if $\Gamma = (U M_w \Phi)^\uparrow_b$ is the rolled path in the generative model, 
then \eqref{eqn:unwrapping:coords:H} implies that $(X_i)_b^{\downarrow \Gamma} = Y_i$,
and, due to \eqref{eqn:Z:as:MN}, 
\begin{equation}
\label{eqn:lik}
H(X_i; \Gamma) = Z_i \sim \mathcal{MN}\left(M_w \Phi, U_w, \Phi^\top V_w \Phi \right),
\end{equation}
that is, the underlying $\{Z_i\}$ in the generative model are exactly recovered. The foregoing steps rely on the $b$ and $U$ in \eqref{eqn:unwrapping:coords:H} being the same $b$ and $U$ in the definition of $\{X_i \sim \mathcal{RMN}(M_w, U_w, V_w; b, U)\}$ in that generative model; however, the following proposition establishes that using alternative $b'$ and $U'$ leads to unwrapping coordinates, and hence model parameters, that differ only equivariantly.
Choose $b'\in M$ and basis $U'$ for $T_{b'}M$, and let $c'$ be a geodesic from $b'$ to $\gamma(0)$.

\begin{proposition}
[Equivariance with respect to $b$, $U$] 
\label{prop:equivariance:in:terms:of:basis}
Let $H'(X_i; \Gamma) = U'^\top (X_i)_{b'}^{\downarrow \Gamma}$, let 
$H(X_i; \Gamma) \sim \mathcal{MN}\left(M_w \Phi, U_w, \Phi^\top V_w \Phi \right)
$, and 
suppose that the constant function is in the span of the basis, $\Phi$.
Then there exists $a \in \mathbb R^d$ and a linear isometry $A : \mathbb R^d \rightarrow \mathbb R^d$ such that
$H' (X_i; \Gamma)= a 1_r^\top + A H (X_i; \Gamma)\sim\mathcal{MN}\left(M'_w \Phi, U'_w, \Phi^\top V_w \Phi \right)$, where
$M_w' = a 1_k^\top + A M_w$ and $U_w'= A U_w A^\top$.
\end{proposition}
Under matrix representations of the parallel transport maps, the operator $A$ is an orthogonal matrix. This Proposition is an analogue of Proposition \ref{prop:equivariance}, specialised to the model \S\ref{sec:parametric:model:for:curves:on:Rd} and to discrete time; it establishes the practical sense in which the choice of $b$, $U$ are inconsequential.


\subsection{Estimation}
\label{sec:estimation}
First we consider the problem in which we are given $\{Z_i\}_{i = 1, \ldots, n}$, each $Z_i$ assumed to be an independent realisation of \eqref{eqn:Z:as:MN}, and the goal is to estimate the parameters $M_w, U_w,$ and $V_w$.
For this, it is helpful to define the right-inverse, $\Phi^- = \Phi^\top(\Phi \Phi^\top)^{-1}$, of $\Phi$ such that $\Phi \Phi^- = I_k$, the $k$-dimensional identity matrix; this right-inverse is unique since the basis $\{\phi_s\}$ defined above is such that $\Phi$ has full row rank, $k$.
Then $\{W_i = Z_i \Phi^-\}_{i=1,...,n}$ are independent realisations from the distribution $\mathcal {MN}(M_w, U_w, V_w)$, for which 
the logarithm of the likelihood function, omitting the constant term not dependent on the parameters, is
\citep{dutilleul1999mle}
\begin{equation}
\ell = 
    -\frac{mn}{2}\log(|U_w|)
    -\frac{kn}{2}\log(|V_w|)
    -\frac{1}{2}\sum_{i=1}^n 
    \operatorname{tr}
    \left\{U_w^{-1}(W_i - M_w) V_w^{-1} (W_i - M_w)^\top \right\}.
    \label{eqn:MN:log:lik}
\end{equation}
The maximum likelihood estimator of the mean matrix is $\hat{M}_w = n^{-1} \sum_{i=1}^n W_i$; this is also the least-squares estimator, that is, the minimiser of
$
M_w \mapsto \sum_{i=1}^n \operatorname{tr}
    \left\{(W_i - M_w) (W_i - M_w)^\top \right\}
$. Maximum likelihood estimators
$\hat{U}_w, \hat{V}_w$ satisfy the pair of equations
\begin{equation}
    \hat{U}_w = \frac{1}{nk} \sum_{i=1}^n (W_i - \hat{M}_w) \hat{V}_w^{-1} (W_i - \hat{M}_w)^\top; \,\hat{V}_w = \frac{1}{nd} \sum_{i=1}^n (W_i - \hat{M}_w)^\top \hat{U}_w^{-1} (W_i - \hat{M}_w),
    \label{eqn:Uw_hat:and:Vw_hat}
\end{equation}
which can be evaluated iteratively until convergence \citep{dutilleul1999mle}. 
We now return to our setting consisting of a sample of discrete-time curves $\{X_i\}_{i = 1, \ldots, n}$, where each $X_i \overset{ind} \sim \mathcal{RMN}(M_w, U_w, V_w; b, U)$ for a fixed point $b \in M$ and basis $U$ of $T_bM$. 
We present multiple estimators of $M_w$, each of which results in different estimators of $U_w$ and $V_w$. The first set of estimates are obtained by maximising a likelihood written using \eqref{eqn:MN:log:lik} for the mean and covariance parameters in \eqref{eqn:lik}; moreover, $H(\Gamma; \Gamma) = M_w \Phi$. 
Alternatively, consider the two estimators: 
\begin{align}
\hat{M}^{\text{LS}}_w &= \underset{M_w}{\operatorname{argmin}}
\sum_{i=1}^n \operatorname{tr}
    \left[\left\{H(X_i; \Gamma)\Phi^- - M_w\right\} \left\{H(X_i; \Gamma)\Phi^- - M_w\right\}^\top \right], \label{eqn:M_w_hat:LS}\\
    \hat{M}^{\text{MLE}}_w &= \underset{M_w}{\operatorname{argmin}}
\sum_{i=1}^n \operatorname{tr}
    \left[\hat{U}_w^{-1} \left\{H(X_i; \Gamma)\Phi^- - M_w\right\} \hat{V}_w^{-1} \left\{H(X_i; \Gamma)\Phi^- - M_w\right\}^\top \right], \label{eqn:M_w_hat:MLE}
\end{align}
in which $\Gamma = \Gamma(M_w)$, that is, in \eqref{eqn:M_w_hat:LS} and \eqref{eqn:M_w_hat:MLE}, since $\Gamma$ is determined by $M_w$, it is implicitly being estimated jointly with $M_w$. The estimators $\hat{M}^{\text{LS}}_w$ and 
$\hat{M}^{\text{MLE}}_w$ require numerical optimisation; in computations we have used the Broyden--Fletcher--Goldfarb--Shannon algorithm as implemented in the \texttt{optim} function in \texttt{R} \citep{R:citation}. Estimators of $U_w$ and $V_w$ are
\begin{align}
    \hat{U}_w &= \frac{1}{nk} \sum_{i=1}^n \left\{H (X_i; \Gamma(\hat{M}_w) )\Phi^-  - \hat{M}_w\right\} \hat{V}_w^{-1} \left\{H(X_i; \Gamma(\hat{M}_w))\Phi^- - \hat{M}_w\right\}^\top, 
    \label{eqn:Uw_hat:in:terms:of:Mw_hat} \\
    \hat{V}_w &= \frac{1}{nd} \sum_{i=1}^n \left\{H(X_i; \Gamma(\hat{M}_w))\Phi^- - \hat{M}_w\right\}^\top \hat{U}_w^{-1} \left\{H(X_i; \Gamma(\hat{M}_w))\Phi^- - \hat{M}_w \right\},
    \label{eqn:Vw_hat:in:terms:of:Mw_hat}
\end{align}
related to \eqref{eqn:Uw_hat:and:Vw_hat}, 
in which $\hat{M}_w \in \{\hat{M}^{\text{LS}}_w, \hat{M}^{\text{MLE}}_w\}$. If $\hat{M}_w$ is taken to be $\hat{M}^{\text{LS}}_w$, then it can be evaluated once at the outset, followed by iterative evaluations of \eqref{eqn:Uw_hat:in:terms:of:Mw_hat} and \eqref{eqn:Vw_hat:in:terms:of:Mw_hat}. If $\hat{M}_w$ is taken to be $\hat{M}^{\text{MLE}}_w$ then, owing to its dependence on $\hat{U}_w$ and $\hat{V}_w$, it becomes necessary to evaluate \eqref{eqn:M_w_hat:MLE}, \eqref{eqn:Uw_hat:in:terms:of:Mw_hat} and \eqref{eqn:Vw_hat:in:terms:of:Mw_hat} iteratively. One reason that $\hat{M}^{\text{MLE}}_w$ might be preferred in spite of this extra computational cost is that in the case where the manifold $M$ has non-positive sectional curvatures then $\hat{M}^{\text{MLE}}_w$ together with $\hat{U}_w$ $\hat{V}_w$ in  (\ref{eqn:Uw_hat:in:terms:of:Mw_hat}, \ref{eqn:Vw_hat:in:terms:of:Mw_hat}) are asymptotically, for large $n$, the maximum likelihood estimators, and are exactly so when $\Gamma$ is the data-generating curve.

Another estimator of $\Gamma$ is based on the discretization of the sample Fr\'echet mean curve, when it exists, which can then be used to estimate $M_w$. This is closely related to the discretization of the estimator $\hat m_r$ in \eqref{eq:rolled_mhat}. Let   
$\hat{\Gamma}_\text{Fre} = \{\hat{\gamma}_\text{Fre}(t_1), \ldots, \hat{\gamma}_\text{Fre}(t_n)\}$ be discrete version of the sample Fr\'echet mean $\hatfre$. An estimator of $M_w$ based on the unwrapped coordinates with respect to $\hat{\Gamma}_\text{Fre}$ is 
\begin{align}
  \hat{M}^{\text{Fre}}_w & = \underset{M_w}{\operatorname{argmin}}
\sum_{i=1}^n \operatorname{tr}
    \left[\left\{H(X_i; \hat{\Gamma}_\text{Fre})\Phi^- - M_w\right\} \left\{H(X_i; \hat{\Gamma}_\text{Fre})\Phi^- - M_w\right\}^\top \right].
    \label{eqn:M_w_hat:Fre}
\end{align}
The estimator $\hat{M}^{\text{Fre}}_w$ is exactly the one described in Remark \ref{rem:param_rates} as the projection of the mean of the unwrapped discrete curves $X_i$ onto the $k$-dimensional subspace spanned by $\Phi$, and has a closed-form expression. 
\begin{proposition}[Estimator $\hat{M}_w^{\text{Fre}}$ via unrolling of $\hatfre$]
\label{prop:mhat_fre}
Under assumption A1  
\[\hat{M}^{\text{Fre}}_w = 
H(\hat{\Gamma}_\text{Fre}; \hat{\Gamma}_\text{Fre}) \Phi^-\thickspace.
\]
\end{proposition}
As discussed in Remark \ref{rem:param_rates}, under assumptions A1-A5, we have $\|\hat{M}^{\text{Fre}}_w-M_w\|_F=O_P(n^{-1/2})$. In practice, especially for data on $M$ with small variability for each $t$,  we found barely any appreciable difference between estimators $\hat{M}^{\text{Fre}}_w$, $\hat{M}^{\text{LS}}_w$, $\hat{M}^{\text{MLE}}_w$.

\subsection{A two-sample test for equality of mean curves}
\label{sec:two:sample:test}
The flattening operation using the unrolling and unwrapping maps enable development of inferential tools for samples of curves on $M$. As a simple illustration, we briefly discuss how this can be done for a two-sample test of equality of means of two samples $\{X_1^{(1)}, \ldots, X_{n_1}^{(1)}\}$ and $\{X_1^{(2)}, \ldots, X_{n_2}^{(2)}\}$  of curves on $M$, which we anticipate will be relevant in many applications.

When $n_1+n_2$ is large enough, it is sensible to consider a test statistic similar to the Hotelling's $T^2$ statistic. This is so, since, under the rolled Gaussian process model, the unwrapping coordinates for the two samples $\{H(X_1,\hat \Gamma_{\text{fre}})^{(1)}, \ldots, H(X_{n_1},\hat \Gamma_{\text{fre}})^{(1)}\}$ and $\{H(X_1,\hat \Gamma_{\text{fre}})^{(2)}, \ldots, H(X_{n_2},\hat \Gamma_{\text{fre}})^{(2)}\}$ are approximately Gaussian distributed, where $\hat \Gamma_{\text{fre}}$ is the discretized sample Fr\'echet mean curve of the pooled sample. If the population version $\Gamma_{\text{fre}}$ of the pooled sample is used in place of $\hat \Gamma_{\text{fre}}$, then the unwrapping coordinates are exactly Gaussian distributed since they are finite-dimensional projections of the covariant Gaussian vector in Proposition \ref{prop:vector_field}. 

Now, suppose

\[
X_i^{(g)}
\sim \mathcal{RMN}(M_w^{(g)}, U_w, V_w; b, U),
\]
in which for $g=1,2$ the mean parameters are possibly different, but the covariance parameters are assumed common. We wish to test $H_0: M_w^{(1)} = M_w^{(2)}$ versus the alternative $H_1: M_w^{(1)} \neq M_w^{(2)}$. From \S\ref{sec:estimation}, given estimates $\hat{M}^{(g)}_w, \hat{U}^{(g)}_w, \hat{V}^{(g)}_w$ for each sample, the analogue of the Hotelling $T^2$ statistic is
\begin{equation}
J = \text{tr} \left\{ \hat{V}_w^{-1} (\hat{M}_w^{(1)} - \hat{M}_w^{(2)})^\top \hat{U}_w^{-1} (\hat{M}_w^{(1)} - \hat{M}_w^{(2)})\right\},
\label{eqn:hotelling:type:stat}
\end{equation}
in which $\hat{U}_w = \{(n_{1} - 1)\hat{U}^{(1)}_w + 
(n_{2} - 1)\hat{U}^{(2)}_w\}/(n_{1} + n_{2} - 2)$ is a pooled estimate of $U_w$,
and $V_w$ is defined similarly. The null hypothesis is rejected for suitably large $J$ with respect to its null distribution, which is accessed by resampling. For a permutation test, the procedure is: pool the two samples as 
$\{X_1^{(1)}, \ldots, X_{n_1}^{(1)}, X_1^{(2)}, \ldots, X_{n_2}^{(2)}\}$ then partition randomly to create resampled data,
$\{X_1^{(1)*}, \ldots, X_{n_1}^{(1)*}\}$ and $\{X_1^{(2)*}, \ldots, X_{n_2}^{(2)*}\}$, say, from which to compute the statistic \eqref{eqn:hotelling:type:stat}. Doing this a large number, say $R$, times gives resampled statistics $J_1^*, ..., J_R^*$ that can be treated as a sample from the null distribution of $J$. The $p$-value can be computed as $p = \{1+\sum_r^R \mathbb I(J_r^* > J)\}/(1+R)$, where $\mathbb I$ is the indicator function. The step of randomly partitioning the pooled data can be replaced by sampling with replacement from the pooled data, to give a bootstrap test. Results later in \S\ref{sec:robot:SO3} use the permutation version described above.

 The two-sample test offers one compelling illustration of the inferential utility of the rolled Gaussian process model for curves observed in discrete time. Evidently, the estimators of $M_w, U_w, V_w$ may be used to develop many other inferential procedures. We also note that various multivariate test statistics that testing equality of the mean may be used; our purpose in this section is to merely demonstrate the utility of one such sensible one.

\section{Practical calculations in embedded coordinates}
\label{sec:calcs:in:embedded:coordinates}

For computations, it is convenient to regard $M$ as a $d$-dimensional embedded Riemannian submanifold of $\mathbb R^q$, for $d < q$, with the induced metric from $\mathbb R^q$; when $d=q$, $M$ is an open submanifold of $\mathbb R^q$, and specific Riemannian metrics may be chosen. Given the global coordinates of $\mathbb R^q$, a point $p$ in $M$ and vector in $T_pM$ are both represented as vectors in $\mathbb R^q$. With such a choice of coordinates, it is natural to identify the tangent space $T_pM$ with the $d$-dimensional affine subspace $\{p+T_pM\}$ in $\mathbb R^q$. 
Thus, in coordinates of $\mathbb R^q$, the unrolling 
$\gamma_b^{\downarrow}$ of $\gamma$ satisfies \eqref{eq:translated:unrolling} subject to initial condition $ \gamma^{\downarrow}_b(0) = b + \exp^{-1}_b\left\{\gamma(0)\right\}$, with solution
\begin{equation} \label{eq:translated:unrolling:soln:embedded}
\gamma^\downarrow_b(t) = b + 
\exp^{-1}_b\{\gamma(0)\} + P^c_{0 \leftarrow 1} 
\left\{\gamma^\downarrow(t) - 
\gamma^\downarrow(0)\right\};
\end{equation}
rolling $\gamma_b^{\uparrow}$ is defined as previously as the unique inverse of unrolling,
and the unwrapping and wrapping maps are as defined before in (\ref{eqn:translated:unwrapping}-\ref{eqn:translated:wrapping}) with elements of $M$ or $T_bM$ being vectors in $\mathbb R^q$, and each parallel transport being a matrix, $\mathbb{R}^{q \times q}$. 
For the rolled Gaussian process $x \sim \mathcal{RGP}(m, K; b, U)$ in embedded coordinates, basis $U = (u_1, \ldots, u_d) \in \mathbb R^{q \times d}$ is a matrix with $d$ orthonormal columns, and the Gaussian process $z \sim \mathcal{GP}(m,K)$ is such that $z(t) = \sum_{i=1}^d e_i z(t)$ where $\{e_i\}$ is the standard basis for $\mathbb R^d$; then for $m_b(t) = b + \sum_{i=1}^d u_i m(t)$ and 
$y: y(t) = b + \sum_{i=1}^d u_i z(t)$, $\gamma = {m}_b^\uparrow$ is the rolled mean, and finally $x = y_b^{\uparrow \gamma}$.

Recall that the discrete-time model \eqref{eqn:discrete:time:model:for:curves:on:M} for curves on $M$ is based upon rolling and wrapping curves from the underlying model \eqref{eqn:Z:as:MN} for curves on $\mathbb R^d$.
Let
$M_b = b1_r^\top + U M_w \Phi$ be the mean curve in \eqref{eqn:Z:as:MN} identified as a curve in $T_bM$, and $\Gamma = (M_b)^\uparrow_b$ be its rolling onto $M$; for $Z$ a
random curve from \eqref{eqn:Z:as:MN}, $Y = b 1_r^\top+ U Z$ is the corresponding
curve in $T_bM$, 
and $X = Y_b^{\uparrow \Gamma}$ its wrapping with respect to $\Gamma$
onto $M$. 

For a curve $X$ on $M$, the curve's unwrapping coordinates \eqref{eqn:unwrapping:coords:H}, with the unwrapping being with respect to the curve $\Gamma$, are
\begin{equation}
    H(X_i ; {\Gamma}) = U^\top \left\{(X_i)_b^{\downarrow {\Gamma}} - b 1_r^\top\right\}.
    \label{eqn:unwrapped:curve:T:embedded}
\end{equation}

\section{Numerical examples}

\begin{figure}
    \centering
    \includegraphics[scale = 0.45, trim = {0cm 0cm 1cm 1.2cm}, clip]{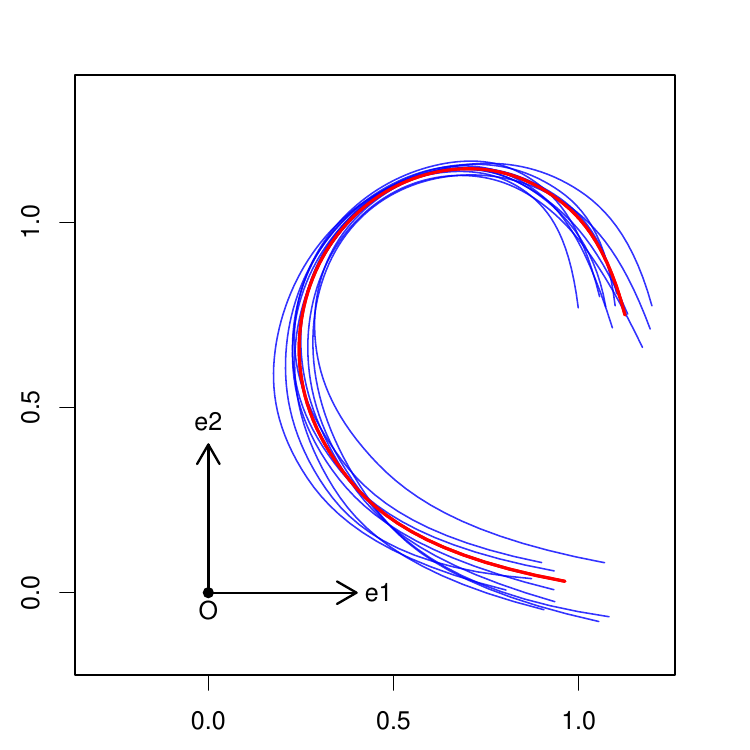}
    \put(-130,135){(a)}
    \includegraphics[scale = 0.33, trim={4cm 5cm 5cm 5.5cm}, clip]{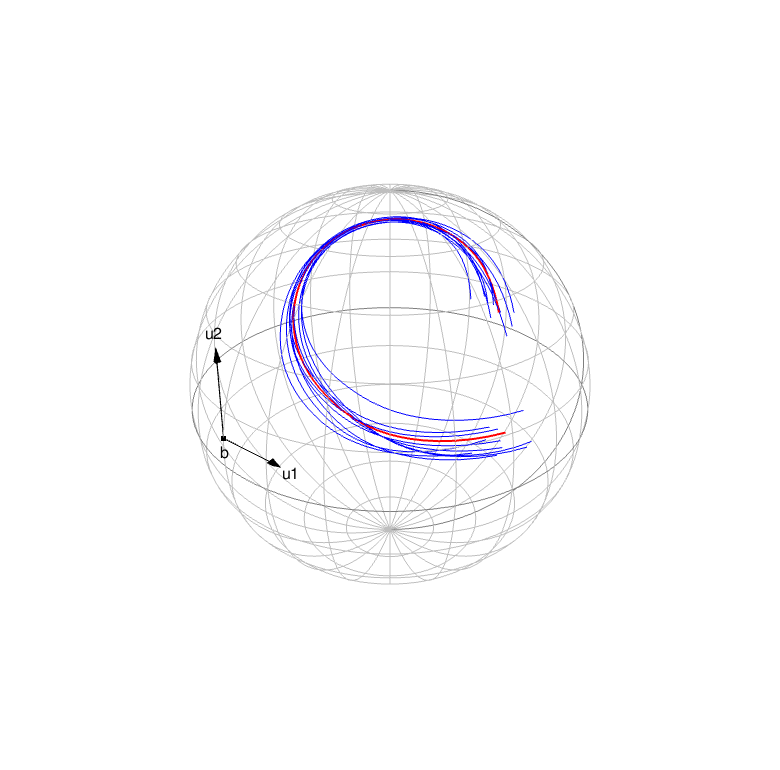}
    \put(-145,135){(b)}
    \caption{(a) Curves from the model in $\mathbb R^2$ described in  \S\ref{example:rolled:model:on:S2}. (b) Corresponding rolled and wrapped curves on $\mathbb S^2$.}
    \label{fig:curves:from:toy:model:R2:S2}
\end{figure}

\subsection{A prescribed model for heteroscedastic curves on $\mathbb S^2$}\label{example:rolled:model:on:S2}
Here we give an example for $M=\mathbb S^2$,
of the steps described in the preceding section to simulate from 
\eqref{eqn:discrete:time:model:for:curves:on:M} for curves on $\mathbb S^2$: first define a model \eqref{eqn:Z:as:MN} in $\mathbb R^2$, with prescribed parameters $M_w, U_w, V_w$; simulate from the model in $\mathbb R^2$; then roll and wrap to produce curves $\mathbb S^2$ that are samples from \eqref{eqn:discrete:time:model:for:curves:on:M}.
Let $s_1, \ldots, s_k$ be an equally spaced sequence of points spanning $[0,1]$, $M_w \in \mathbb R^{2 \times k}$ be the matrix with $i$th column equal to $3/4 \cdot \left[(1,1) + (1/2 + s_i^2/2)\cdot \left\{\cos(5 s_i), \sin(5 s_i) \right\} \right]^\top$, and $k=10$. The mean curve, $m(t) = M_w \phi(t)$, evaluated discretely as $M_w \Phi$ at $r = 100$ values of $t$ equally spaced on $[0,1]$, is shown as the red line in Fig.~\ref{fig:curves:from:toy:model:R2:S2}(a). To prescribe the covariance parameters, let $U_w = I_2$, and let $V_w$ be the matrix with $(i,j)$th element equal to $a_i a_j \rho^{|i-j|}$, with $a_i = 1 + 3/4 \cdot \cos(2 \pi i/k)$ and $\rho = 0.9$. This prescription produces curves in $\mathbb R^2$ with smooth variation and some moderate heteroscedasticity. 
A random curve from this model is
$z_i(t) = W_i \phi(t)$, with $W_i \sim \mathcal {MN}(M_w, U_w, V_w)$,  evaluated numerically as $Z_i = W_i \Phi$.
Fig.~\ref{fig:curves:from:toy:model:R2:S2}(a) shows a sample of $n=10$ curves from this model. 

To roll the model onto $\mathbb S^2$, using the embedding coordinates of $\mathbb R^3$, we fix $b = (51)^{-1} (-5, -5, 1)^\top \in \mathbb S^2$, an arbitrary choice, and fix $U \in \mathbb R^{3 \times 2}$, the basis matrix for $T_b\mathbb S^2$, by taking its columns to be the left singular vectors of $(I_3 - b b^\top)$, excluding the column corresponding to the zero singular value. The rolled mean is $\Gamma = (b 1_r^\top + U M_w \Phi)^\uparrow_b $, and the random curves wrapped with respect to this rolled mean are $X_i = (b 1_r^\top + U Z_i)^{\uparrow \Gamma}_b \sim \mathcal{RMN}(M_w, U_w, V_w; b, U)$.  Figure~\ref{fig:curves:from:toy:model:R2:S2}(b) shows the curves on $M$; red is the rolled mean, $\Gamma$, and blue are the wrapped curves, $\{X_i\}$.

\subsection{Curves on $\operatorname{Sym}_{> 0}(2)$}\label{example:rolled:model:on:SPD}
We equip the three-dimensional manifold $\operatorname{Sym}_{> 0}(2)$ with the affine-invariant metric, $d(\cdot, \cdot)$, which makes it negatively curved \citep[e.g.][Chapter 3]{pennec2019riemannian}; see \ref{sec:appdx:expressions:for:SPD} for definition of $d(\cdot, \cdot)$. The global coordinates of $\mathbb  R^3$ is used to define a model with parameters $M_w, U_w, V_w$ and roll onto $\operatorname{Sym}_{> 0}(2)$, and then infer the parameters from simulated data.
With $s_1, \ldots, s_k$ an equally spaced sequence of points spanning $[0,1]$ and $k=5$, let
$M_w \in \mathbb R^{3 \times k}$ be the matrix with $i$th column equal to $3/4 \cdot \left[
(1/5)\cdot \left\{\cos(5 s_i), \sin(5 s_i), s_i \right\} \right]^\top$, $U_w = I_3$, and $V_w$ be the matrix with $(i,j)$th element equal to $10^-3 a_i a_j \rho^{|i-j|}$, with $a_i$, $\rho$ as in \S \ref{example:rolled:model:on:S2}. We take the tangent point $b = I_2$ and basis for tangent space $U = I_3$, then simulate data $\{X_i\}$ as in \S\ref{example:rolled:model:on:S2}. Table \ref{table:spd} shows numerical results from simulating the data then fitting the model to the simulated data, to investigate convergence of estimators $\hat{M}_w = \hat{M}_w^\text{Fre}$ from \eqref{eqn:M_w_hat:Fre} and $\hat{U}_w$ and $\hat{V}_w$ from 
\eqref{eqn:Uw_hat:in:terms:of:Mw_hat} and \eqref{eqn:Vw_hat:in:terms:of:Mw_hat} to their data-generating values. Results in the table show each estimator, with respect to the indicated metric, getting closer to  towards its population counterpart as the sample size, $n$, increases; a Mahalaonobis-type metric is used for $\hat{M}_w $ and $d(\cdot,\cdot)$ is used for $\hat{U}_w$ and $\hat{V}_w $.

\begin{table}
\def~{\hphantom{0}}
{%
\begin{tabular}{llccccc}
\hline
Estimator & Metric & \multicolumn{5}{c}{Sample Size, $n$} \\
\hline
      & & 10 & 25 & 50 & 100 & 500 \\
$\hat{M}_w$ & 
$\operatorname{tr}
    \{U_w^{-1}(\hat{M}_w - M_w) V_w^{-1} (\hat{M}_w - M_w)^\top \}$ & 2.19 & 0.46 & 0.30 & 0.14 & 0.10 \\
$\hat{U}_w$ & $d(\hat{U}_w, U_w)$ & 0.47 & 0.27 & 0.19 & 0.15 & 0.06 \\
$\hat{V}_w$ & $d(\hat{V}_w, V_w)$ & 1.16 & 0.72 & 0.55 & 0.32 & 0.15 \\
\end{tabular}}
\caption{Results of numerical investigation into convergence of parameter estimates with increasing $n$, for the $\operatorname{Sym}_{> 0}(2)$ model in \S\ref{example:rolled:model:on:SPD} }
\label{table:spd}
\end{table}


\section{Application: orientation of end-effector of a robot arm, curves on $SO(3)$}\label{sec:robot:SO3}

\begin{figure}
    \centering
    \includegraphics[scale=0.3, trim={3.5cm 4.5cm 23.5cm 5cm}, clip]{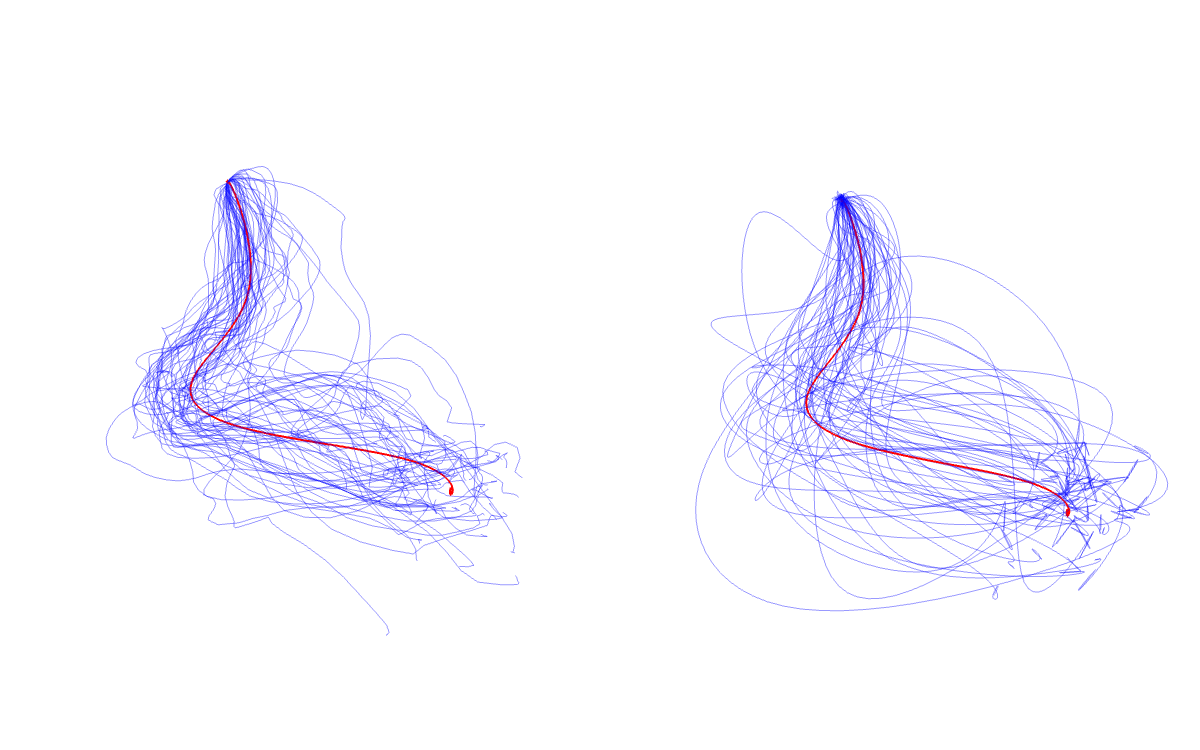}
    \put(-125,105){(a)}
    \includegraphics[scale=0.3, trim={24cm 4.5cm 0.8cm 4.55cm}, clip]{art/SO3_curves_in_R3_projected.png}
    \put(-150,105){(b)}
    \includegraphics[scale = 0.65]{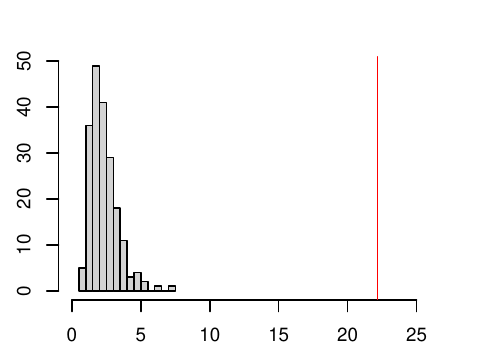}
    \put(-138,105){(c)}
    \caption{For robot SO(3) curves described in \S\ref{sec:robot:SO3}: (a) unwrapped data, blue, and unrolled fitted mean, red; (b) simulations from the fitted Gaussian process model; (c) bootstrap null distribution for test described in text.}
    \label{fig:robot:SO3}
\end{figure}

The data are orientations of the end-effector of a Franka robot arm as it was guided $n=60$ times by the third author to perform a task to deposit the contents of a dustpan into a bin. The orientation can be represented by an element of ${SO}(3)$, the set of $3$-by-$3$ orthogonal matrices with unit determinant, that describes its rotation from an initial reference orientation. An element of $SO(3)$ may be parameterised as an unsigned unit quaternion, and $SO(3)$ is hence identified with $\mathbb S^3$ modulo the antipodal map \citep{prentice1986orientation}. Fixing the sign of each data point then identifies $SO(3)$ with a hemisphere of $\mathbb S^3$, which is our manifold $M$ equipped with the induced geometry of $\mathbb S^3$, whose geometry is inherited from $\mathbb R^4$. 

The curves, $X_i$, are recorded at $r = 100$ time points. We choose basis dimension $k = 10$, and set $b = e_1$ and $U = (e_2, e_3, e_4)$, where $\{e_i\}$ is the standard basis for $\mathbb R^4$. 
We compute the estimated mean curve $\hat{\Gamma} = (b1_r^\top + U \hat{M}_w \Phi)_b^\uparrow$, with $\hat{M}_w$ estimated as in \eqref{eqn:M_w_hat:LS}. Then the unwrapping coordinates, $H_i = H(X_i; \hat{\Gamma})$ for $H(\cdot; \cdot)$ in \eqref{eqn:unwrapping:coords:H}, are curves in $\mathbb R^3$; these are shown as blue curves in Fig.~\ref{fig:robot:SO3}(a), together with corresponding unrolled mean, $H(\hat{\Gamma}; \hat{\Gamma})$, 
plotted in red.
The curves have a common starting point at $t=0$, shown near the top-left in this plot; for small $t$ the variability is small, but then at larger $t$ there is much larger variability, especially following a kink point that corresponds to the dustpan being turned to empty its contents. We further estimate the covariance matrices $U_w$ and $V_w$ in \eqref{eqn:discrete:time:model:for:curves:on:M} using (\ref{eqn:Uw_hat:in:terms:of:Mw_hat}, \ref{eqn:Vw_hat:in:terms:of:Mw_hat}). As a visual appraisal of the fitted model, Fig.~\ref{fig:robot:SO3}(b) shows $n=60$ realisations from $\mathcal{RMN}(\hat{M}_w, \hat{U}_w, \hat{V}_w; b, U)$,  using the same unwrapping coordinates and projection as in Fig.~\ref{fig:robot:SO3}(a). These simulated curves are smoother than the real data, which is a consequence of the basis used, but they have similar 
heteroscedastic variation.

A feature of these data is that they were collected over two sessions, so the data are in fact two samples each of $n^{(g)} = 30$ curves, where $g \in \{ 1,2\}$ indexes the sample. An interesting statistical question is whether the two samples have different statistical characteristics. We perform the test described in \S\ref{sec:two:sample:test} to test the null hypothesis that two samples arise from populations with equal mean curve.
Figure \ref{fig:robot:SO3}(c) 
shows a histogram of values of test statistic $J$ in \eqref{eqn:hotelling:type:stat} simulated under the null hypothesis of equal means, using the permutation procedure described in \S\ref{sec:two:sample:test} with $R=200$ resamples, and the red line shows the observed value of the statistic $J = 22.17$ computed from the data. This observed value is extreme compared with the null distribution, providing conclusive evidence against the null hypothesis of equal means.

\section{Discussion}
Sample path and distributional properties of Euclidean Gaussian processes have been extensively studied. There is much to investigate about analogous properties of the rolled Gaussian process, and this constitutes ongoing work. Moreover, the rolling approach  presented here extends beyond the Gaussian process setting, and in future work we will explore such models. 

For our examples, we leveraged closed-form expressions for exponential, inverse exponential, and parallel transport maps when computing with the four maps in \S\ref{sec:four:maps}. On manifolds lacking these closed-form expressions, approximations via retractions and their inverses offer practical alternatives \citep{absil2009optimization}. See also \cite{nguyen2025parallel} for matrix manifolds. 

Within the Gaussian process setting of this paper, we have specialised to a particular parameterization of the mean and covariance, and assumed common and equispace observation times; these provide modelling and computational convenience but are straightforward to relax. When sample curves are observed sparsely at irregularly spaced observation times, and perhaps in the presence of measurement error, some form of smoothing or resampling of the curves is required to ensure that the design points are common to all curves. Once this is done, the proposed estimators of the mean $m$ and $K$ covariance functions of the Euclidean Gaussian process, defined using unwrapping coordinates \S\ref{sec:estimation} may be used. In future work, we will develop the corresponding asymptotic theory for such estimators. 

 Relatedly, a further generalization is to incorporate a model of nuisance variation in the time variable via warping. This would accommodate the possibility of time registration of observed curves, often valuable in analogous Euclidean settings \citep{zhang2018phase, SSKS}. However, for rolled models, warping combined with rolling raises challenges for computation and inference that remain to be resolved.

\section{Acknowledgements}
We thank Huiling Le for helpful discussions. SP acknowledges support from a grant from the Engineering and Physical Sciences Research Council (EPSRC EP/T003928/1).  KB acknowledges support from grants from the Engineering and Physical Sciences Research Council (EP/Z003377/1; EPSRC EP/X022617/1), National Science Foundation (NSF 2015374), and the National Insitutes of Health (NIH R37-CA21495). 

\appendix


\section{Appendix}
\subsection{Gaussian distribution on inner product spaces}
\label{sec:Gaussian}
We briefly review the concepts required for our purposes and refer to \cite{EatonBook} for more details. Let $(X,\langle \cdot, \cdot \rangle_X)$ be a $d$-dimensional vector space. With respect to a probability space $(\Omega, \mathcal F, \mathbb P)$ if a random vector $Z$ in $X$ is such that $E(\langle x, Z \rangle_X)<\infty$ for every $x \in X$, there there exists a unique vector $\mu \in X$, known as the \emph{mean} of $Z$, such that $E(\langle x, Z \rangle_X)=\langle \mu, Z \rangle_X$; we can employ the notation $\mu=EZ$. 

In similar fashion, with $\text{Cov}$ as the usual covariance between real-valued random variables, if $E(\langle x, Z \rangle_X^2)<\infty$ for every $x\in X$, the unique non-negative linear operator $\Sigma:X \to X$ that satisfies 
\[
\text{Cov}(\langle x, Z \rangle_X, \langle y, Z \rangle_X)=\langle x, \Sigma \thinspace y \rangle_X, \quad \forall \thickspace x,y \in X,
\]
is referred to as the \emph{covariance} of $Z$. Let $(Y,\langle \cdot, \cdot \rangle_Y)$ be another $d$-dimensional inner product space, and $\mathcal L(X,Y)$ be the space of bounded linear operators from $X$ to $Y$. Given a random vector $Z$ in $X$ with mean $\mu$ and covariance $\Sigma$, the random vector $CZ$ for $C \in \mathcal L(X,Y)$ has mean $C\mu$ and $C\Sigma C^\top$, where $C^\top \in \mathcal L(Y,X)$ is the unique adjoint of $C$ which satisfies $\langle b, Ca \rangle_Y= \langle C^\top b,a \rangle_X$ for all $a \in X$ and $b \in Y$.

The following provides a basis-free definition of a Gaussian distribution an inner product space.
\begin{definition}
 A random vector $Z \in (X,\langle \cdot,\cdot \rangle_X)$ has a Gaussian distribution with mean $\mu$ and covariance $\Sigma$ if for every $x \in X$ the real-valued random variable $\langle x,Z \rangle$ has a Gaussian distribution on $\mathbb R$ with mean $\langle x,\mu \rangle_X$ and variance $\langle x, \Sigma x \rangle_X$. 
\end{definition}

With respect to $(\Omega, \mathcal F, \mathbb P)$ suppose $X_1$ and $X_2$ are two random vectors on $(X_1,\langle \cdot, \cdot \rangle_{X_1})$ and $(X_2,\langle \cdot, \cdot \rangle_{X_2})$, respectively, with covariances $\Sigma_1$ and $\Sigma_2$. The random vector $(X_1,X_2)$ then assumes values in the direct sum space $X_1 \oplus X_2$ equipped with the inner product $\langle \cdot, \cdot \rangle_{X_1}+\langle \cdot, \cdot \rangle_{X_2}$. A well-defined bilinear map $(x_1,x_2) \mapsto \text{Cov}(\langle x_1, X_1 \rangle_{X_1}, \langle x_2, X_2 \rangle_{X_2})$ that is positive semidefinite ensures the existence of a unique operator $\Sigma_{12} \in \mathcal L(X_1,X_2)$. Then, the linear operator $\Sigma \in \mathcal L(X_1 \oplus X_2,X_1 \oplus X_2)$ defined as $\Sigma(x_1,x_2)=(\Sigma_{11} x_1+\Sigma_{12}x_2,\Sigma_{12}^\top x_1,\Sigma_{22}x_2)$, is known as the covariance of the random vector $(X_1,X_2)$ in $X_1 \oplus X_2$. The above definition of extends easily to random vectors $(X_1,\ldots,X_m)$ with values in $X_1 \oplus \cdots \oplus X_m$. 

\subsection{Orthonormal frame along a curve $\gamma$}
\label{app:ortho_frame}
Denote by $\mathfrak X(\gamma)$ the set of vector fields along 
 a curve $\gamma:[0,1] \to M$. The intrinsic description requires a curve $\gamma:[0,1] \to M$ and an orthonormal frame $E^\gamma=(e_1,\ldots,e_d)$ with $e_i \in \mathfrak X(\gamma), i=1,\ldots,d$ such that $E^\gamma(\gamma(t))=(e_1(\gamma(t)),\ldots,e_d(\gamma(t)))$ is an orthonormal basis of $T_{\gamma(t)}M$ for every $t$. Simplify notation using $E^\gamma(t)=(e_1(t),\ldots,e_d(t))$ to denote the frame, and note that $E^\gamma$ is obtained by choosing a basis $(v_1,\ldots,v_d)$ for $T_{\gamma(0)}M$, and setting $E^\gamma(t):=(P^\gamma_{t\leftarrow 0}v_1,\ldots,P^\gamma_{t\leftarrow 0}v_d)$ to be a parallel frame with $e_i(t)=P^\gamma_{t\leftarrow 0}v_i$.  Then every vector field $V \in \mathfrak X(\gamma)$ may be expressed as $V(t)=\sum_{i=1}^d V^i(t)e_i(t)$, where $t \mapsto V^i(t) \in \mathbb R$ are component functions. 

If instead we start with an orthonormal basis $U=(u_1,\ldots,u_d)$ of the tangent space $T_bM$ of a point $b \in M$, the orthonormal frame $E^\gamma$ is obtained by first choosing $v_i=P^c_{1\leftarrow 0}u_i$, where $P^c_{t \leftarrow s}:T_{c(s)}M \to T_{c(t)}M$ is the parallel transport along the (unique) geodesic $c:[0,1]\to M$ with $c(0)=b$ and $c(1)=\gamma(0)$, so that
\[
E^\gamma(t)=(P^\gamma_{t\leftarrow 0}P^c_{1\leftarrow 0}u_1,\ldots,P^\gamma_{t\leftarrow 0}P^c_{1\leftarrow 0}u_d), \quad t \in [0,1].
\]
In other words, the vector fields $e_i:=P^\gamma_{t\leftarrow 0}P^c_{1\leftarrow 0}u_i,i=1,\ldots,d$ that constitute $E^\gamma$ are parallel vector fields in $\mathfrak X(\gamma)$ along $\gamma$ determined by $U$. Note that $E^\gamma(0)$ depends on the curve $\gamma$ only through $\gamma(0)$ via the geodesic $c$ connecting $b$ to $\gamma(0)$. Thus for each $t$, $E^\gamma(t): T_b M \to T_{\gamma(t)}M$ is a linear isomorphism engendered by its definition using parallel transport maps on the basis $U$ of $T_bM$.


\subsection{$C^k$ curves in $M$}
\label{app:c_k curves}
Given a chart $(U,\psi)$ with open $U \subset M$ and smooth $\psi:U \to V$, where $V$ is an open subset of $\mathbb R^d$, the coordinate representative of a curve $c:[0,1] \to M$ is the Euclidean curve $\tilde c=\psi \circ c$. Then $c$ is a $C^1$ curve if $\tilde c$ is $C^k$ with respect to the norm $\|\tilde c\|_{C^k}:=\sum_{j=0}^k \sup_t \|D^j\tilde c(t)\|$, where $D^j$ is the $j$th derivative operator with $D^0$ equalling the identity. Then, a sequence of curves in $M$ converge in the weak $C^k$ topology to a limit if their coordinate representatives converge to the corresponding representative for the limit with respect to the $C^k$ norm \citep[Chapter 2][]{hirsch2012differential}. 
\section{Numerical implementation}

\subsection{Computations with discrete representations of curves} \label{sec:computations:for:rolling:etc:in:discretised:time}

Numerical implementation necessarily entails discrete representations of continuous curves; 
for example, the base curve $t \rightarrow \gamma(t)$ that determines rolling is stored as
$\Gamma = \{\gamma(t_1), \ldots, \gamma(t_m)\}$. This discrete representation, $\Gamma$, can be regarded as a piecewise geodesic approximation of $\gamma$ that coincides with $\gamma$ at $t_1, \ldots, t_m$ and that can be made arbitrarily accurate by making the time discretization arbitrarily fine.  Define $\tilde \gamma$ as a piecewise geodesic curve with geodesic segments connecting consecutive elements of $\Gamma$, to approximate $\gamma$ in computations. Then, for example, parallel transport along the discretized curve, $\Gamma$, is
\[
P^{\tilde\gamma}_{t \leftarrow t_0} =
P^{\tilde\gamma}_{t \leftarrow t_{j}} \cdots P^{\tilde\gamma}_{t_2 \leftarrow t_1} P^{\tilde \gamma}_{t_1 \leftarrow t_0},
\]
here for
$t$ satisfying $t_{j} \leq t \leq t_{j+1}$; in other words, the parallel transport is a composition of parallel transports each along a geodesic, enabling simple computation. 
To compute the discretized unrolling, $\Gamma^\downarrow = \{\gamma^\downarrow(t_1), \ldots, \gamma^\downarrow(t_m)\}$, of 
$\Gamma$, from \eqref{eqn:unrolling:defn} and with $\tilde\gamma$ replacing $\gamma$, 
\begin{align}
   \gamma^\downarrow(t) &= \gamma(t_0) + \left(\int_{t_0}^{t_1} + \cdots + \int_{t_{j}}^t \right) P^{\tilde\gamma}_{t_0 \leftarrow s} \dot{\tilde \gamma}(s)  \, \mathrm{d}s \nonumber
   \\ & = \gamma^\downarrow(t_{j}) + \int_{t_j}^t
   P^{\tilde\gamma}_{t_0 \leftarrow t_j} P^{\tilde\gamma}_{t_{j} \leftarrow s} 
   \dot{\tilde \gamma}(s) \mathrm{d} s  = 
   \gamma^\downarrow(t_{j}) + P^{\tilde\gamma}_{t_0 \leftarrow t_j} \exp^{-1}_{\gamma(t_j)} {\tilde\gamma}(t),  
   \label{eqn:unrolling:computational:expression}
\end{align}
using recursion, 
and using that along the geodesic segment between ${\tilde\gamma}(t_j)$ and ${\tilde\gamma}(t)$,
$
\int_{t_j}^t
 P^{\tilde \gamma}_{t_{j} \leftarrow s} 
   \dot{\tilde \gamma}(s) \mathrm{d} s = t \dot{\tilde \gamma}(t_j) = 
   \exp^{-1}_{\gamma(t_j)} \tilde{\gamma}(t).
$
Hence, via equation \eqref{eqn:unrolling:computational:expression}, 
$\Gamma^\downarrow = \{\gamma^\downarrow(t_1), \ldots, \gamma^\downarrow(t_m)\}$ can be determined from $\Gamma = \{\gamma(t_1), \ldots, \gamma(t_m)\}$. Similarly, via the rearrangement of  \eqref{eqn:unrolling:computational:expression} as
\begin{equation}
\gamma(t) = \exp_{\gamma(t_j)} \left[ P^{\tilde \gamma}_{t_j \leftarrow t_0} \left\{ \gamma^\downarrow(t) - \gamma^\downarrow(t_j)\right\} \right],
\label{eqn:rolling:computational:expression}
\end{equation}
thus $\Gamma$ can be determined from
$\Gamma^\downarrow$.

\subsection{Expressions for manifold $\mathbb S^d$}

The embedded representation of the manifold is 
$\mathbb S^d=\{x \in \mathbb R^{q} : x^\top x = 1
\}$ for $q = d+1$.  Suppose that $p, p', x \in M$; $y \in T_pM$; and $c$ is a geodesic curve connecting $p$ to $p' \in M \backslash \mathcal C(p)$ such that $c(0) = p$ and $c(1) = p'$. Then \citep[for example]{absil2009optimization},
\begin{align*}
    \exp_p^{-1}(x) & = u \{x - p \cos(u)\}/\sin(u) \in T_pM, \quad x \notin \mathcal C(p); \\
    \exp_p(y) & = p \cos(\|y \|) + (y / \| y \|) \sin(y/ \| y \|) \in M; \\
    P_{1 \leftarrow 0}^c y & = 
    \left[ I_d + \{\cos(v) - 1\} w w^\top -
    \sin(v) p w^\top \right] y \in T_{p'}M;
\end{align*}
for $u = \cos^{-1}(p^\top x)$, $v = \cos^{-1}(p^\top p')$, $w = \{p' - p  \cos(v)\}/\sin(v)$.





\subsection{Expressions for manifold $\operatorname{Sym}_{> 0}(d)$.}

\label{sec:appdx:expressions:for:SPD}

Let
{$M = \operatorname{Sym}_{> 0}(r) =\{X \in \mathbb R^{r \times r}: X= X^\top, y^\top X y>0, y \in \mathbb R^r\}$ be the positive cone of dimension $d=r(r+1)/2$ within the set of $r$-dimensional symmetric matrices.}
Suppose that $P, P', X \in M$; $Y \in T_pM$; and $c$ is a geodesic curve connecting $P$ to $P'$ such that $c(0) = P$ and $c(1) = P'$.
The distance from the affine-invariant metric \citep[for example]{pennec2019riemannian} is
\[
d(P,X)= \|\log(P^{-\frac{1}{2}} X P^{-\frac{1}{2}})\|,
\]
and 
\begin{align*}
    \exp_P^{-1}(X) & = P^{1/2} \log\left( P^{-1/2} X P^{-1/2}\right) P^{1/2} \in T_PM;\\
    \exp_P(Y) & = P^{1/2} \exp\left( P^{-1/2} Y P^{-1/2}\right) P^{1/2} \in M;\\
    P_{1 \leftarrow 0}^c Y & = 
    E Y E^\top \in T_{P'}M,
\end{align*}
with $E = (P' P^{-1})^{1/2}$ \citep{sra2015conic},
and $\log$ and $\exp$ are the matrix logarithm and exponential.

\section{Supporting lemmas}
The results in the section will be used in the proofs of the main results in Appendix D. We separately record results pertaining to the geometry from those that relate to the statistical model. Recall that $M$ is a complete, connected $d$-dimensional Riemannian manifold. 
\subsection{Geometric lemmas}
\begin{lemma}
\label{lemma:Lipschitz}
Let $A \subset M$ be compact with bounded geometry satisfying assumption A4. 
There exists a set $\Omega \subset A \times A$ such that $(p,q) \mapsto \exp^{-1}_p(q)$ is well-defined within $\Omega$, and is smooth within the closure of $\Omega$ with derivatives uniformly bounded by constants that depend only on the local geometry of $\Omega$.
\end{lemma}
\begin{proof}
The function $p \mapsto \text{inj}(p)$ is continuous on $M$ \citep[e.g.][Theorem 5.7]{petersen2016symmetric}, and attains its minimum within the compact $A$. Since $M$ is complete and $A \subset M$, $\text{inj}(A)>0$ \citep[e.g.][Lemma 6.16]{lee2018introduction}, and it is possible to choose an $r_0>0$ with
\[
r_0 < \min\left\{\frac{\pi}{2 \sqrt{\kappa}},\frac{1}{2}\text{inj}(A)\right\}
\]
such that the ball $B_{p^*}(r_0)$ of radius $r_0$ around some point $p^* \in A$ is geodesically convex \footnote{For any $p,q$ within $B_{p^*}(r_0)$ there exists a unique minimizing geodesic $\gamma_{pq}$ contained within $B_{p^*}(r_0)$ connecting $p$ and $q$.}. Let $S:=\{(p,v): p \in B_{p^*}(r_0), \|v\|_{p*} < r_0 \} \subset TM$, and define the map 
\[
\mathcal E:S \to M \times M, \quad (p,v) \mapsto \mathcal E(p,v):=(p,\exp_p (v)),
\]
with $\mathcal E(S) \subseteq B_{p^*}(r_0) \times B_{p^*}(r_0)$. Note that $\exp_p$ is a smooth diffeomorphism within $B_p(r_0)$, and thus if $\mathcal E(p,v)=(p',q')$ and $\mathcal E(p',v')=(p',q')$ with $\|v\|_p, \|v'\|_{p'} < r_0$ we must have $p=p'$ and $v=v'$; $\mathcal E$ is thus injective and its differential is invertible everywhere on $B$. By the inverse function theorem, $\mathcal E$ is a diffeomorphism onto its image 
$\Omega:=\mathcal E(S)$. This ensures that the $(p,q) \mapsto \exp^{-1}_p(q)$ is well defined with $\Omega$. 

The set $\Omega$ has a compact closure within $M \times M$ since it can be covered by a union of closed balls around points in the closure of $B_{p^*}(r_0)$, which is compact. Thus within the closure of $\Omega$, $(p,q) \mapsto \exp^{-1}_p(q)$ has uniformly bounded derivatives of all orders in $p$ and $q$. Using normal coordinates at $p$ it is easily seen that derivatives in $q$ depend only on the upper bound $\kappa$ on sectional curvatures in $A$ and the radius $r_0$. The first derivative in $p$ is the covariant derivative of the vector field $p \mapsto \exp^{-1}_p(\cdot)$, given by the value of a particular Jacobi field which is uniformly bounded by constants that depend only on $\kappa$ and $r_0$; see, for example, Appendix A in \cite{karcher1977riemannian} for details, and also Lemma 1 in \cite{reimherr2021differential}. 
\end{proof}

\begin{lemma}
 \label{lem:tubular}
  There exists an $r_0>0$ such that within the tubular neighbourhood $\mathbb T_{\gamma}(r_0)$ of a continuous curve $\gamma$ the map $(p,q) \mapsto \exp^{-1}_p(q)$ is well-defined with bounded derivatives. 
 \end{lemma}
\begin{proof}
Due to Lemma \ref{lemma:Lipschitz}, it suffices to prove that $\mathbb T_{\gamma}(r)$ for any finite $r>0$ is a compact subset of $M$. Then, $r_0$ may be chosen as in the proof of Lemma \ref{lemma:Lipschitz}. The set $T_{\gamma}(r_0)$ is closed and bounded, and is compact since $M$ is complete. 
\end{proof}
The curvature tensor $R: T_pM\times T_pM\times T_pM\to T_pM$ is endowed with a pointwise operator norm bound
\begin{equation}
\label{eq:R_norm}
\|R\|_{\text{C}}:=\sup_{\substack{u,v\in T_pM\\ \|u\|_p=\|v\|_p=1}} \| R(u,v) \|_{\text{op}}.
\end{equation}


\begin{lemma}[Holonomy bound for triangles]
\label{lemma:holonomy}
Let $p,q,r$ be distinct points in $S \subset M$ with $|\sec{(S)}| \leq \kappa$ for $\kappa \geq 0$. Define the operator $\Xi:=P^{\sigma_{qr}}_{0 \to 1}P^{\sigma_{pq}}_{0 \to 1}-P^{\sigma_{pr}}_{0 \to 1}:T_pM \to T_rM$, where $\sigma_{pq}:[0,1] \to M$ is a curve from $\sigma_{pq}(0)=p$ to $\sigma_{pq}(1)=q$, and similarly for the other two curves. Then, 
\[
\|\Xi\|_{\text{op}} \leq C \kappa \rho(r,p)\rho(q,r)+o(A(\Delta_{pqr})),
\]
where the constant $C$ depends on dimension $d$, and $A(\Delta_{prq})$ is the area of the triangle $\Delta_{prq}$ formed by $p,q$ and $r$. 
\end{lemma}
\begin{proof}
Observe that 
\[
\|\Xi\|_{\text{op}}=\|P^{\sigma_{qr}}_{0 \to 1}P^{\sigma_{pq}}_{0 \to 1}-P^{\sigma_{pr}}_{0 \to 1}\|_{\text{op}}=\|P^{\sigma_{qr}}_{0 \to 1}P^{\sigma_{pq}}_{0 \to 1}P^{\sigma_{pr}}_{1 \leftarrow 0}-I\|_{\text{op}},
\]
since $P^{\sigma_{pr}}_{1 \leftarrow 0}$ is an isometry, where $P^{\sigma_{pr}}_{1 \leftarrow 0}=(P^{\sigma_{pr}}_{0 \to 1})^{-1}$. The operator $\Xi$ is the holonomy operator for parallel transport along the triangle $\Delta_{pqr}$. 
For the family of curves
$\gamma : [0,1]\times[0,1] \to M$ where $\gamma(s,\tau)$ is the point at time $\tau$ along the curve $\sigma_{pq}(s)$ to $r$ so that for each fixed \(s\),  \(\gamma(s,\cdot)\) is the curve from \(\gamma(s,0)=\sigma_{pq}(s)\) to \(\gamma(s,1)=r\); note that $\gamma(0,\cdot)=\sigma_{pr}$  and $\gamma(1,\cdot)=\sigma_{qr}$. Then, for any unit vector $w \in T_pM$
\[
(P^{\sigma_{qr}}_{0 \to 1}P^{\sigma_{pq}}_{0 \to 1}P^{\sigma_{pr}}_{1 \leftarrow 0}-I) w =\left(\int_{\Delta_{pqr}} R(\partial_s \Delta_{pqr}, \partial_\tau \Delta_{pqr}) \thickspace \text{d}s\text{d}\tau \right)w+w \delta_{pqr},;
\]
where $\delta_{pqr}=o(A(\Delta_{pqr}))$; see, for example, Theorem 7.11 in \cite{lee2018introduction}. It is known that $\|R\|_{\text{C}}$ is upper bounded by a dimension-dependent constant times supremum of the absolute sectional curvatures \citep{petersen2016symmetric}. Thus, taking norms on both sides, and noting that the area of a triangle is upper bounded by the product of lengths of any two of its sides, the proof follows. 
\end{proof}
\subsection{Probabilistic lemmas}
We first collect some results on the Fr\'echet functional $M \times [0,1] \ni (p,t)\mapsto F(p,t):=\mathbb E(\rho^2(p,x(t)))$ from \eqref{Eqn:FM_definition}, its covariant derivative $\nabla_p F(p,t)$, its Hessian $\nabla^2_p F(p,t)$, covariant time derivatives $\nabla_t F(p,t)$ and $\nabla^2_t F(p,t)$, and mixed derivatives $\nabla_t \nabla_p F(p,t)$ and $\nabla_p \nabla_t F(p,t)$. 

We will use the following relations related to the local isometry property of the rolling map. Let $\gamma=m^\uparrow$ and $e(t):=z(t)-m(t)$, a mean-zero Gaussian process in $\mathbb R^d$. The vector field $V^{\gamma}$ along $\gamma$ in Proposition \ref{prop:vector_field} satisfies
$\|V^{\gamma}(t)\|_{\gamma (t)}=\|y_b(t)-m_b(t)\|_b=\|e(t)\|$ for each $t$. similarly,  $\|\nabla_t V^{\gamma}(t)\|_{\gamma (t)}=\|\dot e(t)\|$, and $|\dot \gamma(t)\|_{\gamma(t)}=\|m(t)\|$.

\begin{lemma}
\label{lemma:FrechetFunction_deriv}
Let $x \sim \mathcal{RGP}(m,k;b,U)$ satisfying assumption A2-A4. Then, on $S$,
\begin{enumerate}
    \item [(a)] $\nabla_p F(p,t)=\mathbb E \left(\nabla_p \rho^2(p,x(t))\right)$ and $\nabla^2_p F(p,t)=\mathbb E \left(\nabla^2_p \rho^2(p,x(t))\right)$;
    \item[(b)] $\nabla_tF(p,t)=\mathbb E \left(\nabla_t \rho^2(p,x(t))\right)$ and $\nabla^2_tF(p,t)=\mathbb E \left(\nabla^2_t \rho^2(p,x(t))\right)$;
    \item[(c)]$\nabla_t \nabla_p F(p,t)=\mathbb E \left(\nabla_t \nabla_p\rho^2(p,x(t))\right)$ and $\nabla_p\nabla_t F(p,t)=\mathbb E \left(\nabla_p\nabla_t \rho^2(p,x(t))\right)$.
\end{enumerate}
\end{lemma}
\begin{remark}
Assertions of Lemma \ref{lemma:FrechetFunction_deriv} hold more generally for any second-order process $z$, and also when $\gamma=m^\uparrow$ does not equal the Fr\'echet mean curve $\fre$. 
\end{remark}
\begin{proof}
For each $t$, outside of the cut locus $\mathcal C(x(t))$ of $x(t)$ the map $p \mapsto \rho^2(p,x(t))$ is smooth with well-defined derivatives of all orders. The random variable $x(t)$ has a distribution that is the pushforward under $\exp_{\gamma(t)}:T_{\gamma(t)}M \to M$ of a $d$-dimensional Gaussian, and is absolutely continuous with respect to the volume measure on $M$. Thus, $\mathbb P(x(t) \in \mathcal C(p))=0$, since $\mathcal C(p)$ for every $p$ is of dimension $d-1$. The $p \mapsto \rho^2(p,x(t))$ thus is smooth a.e. for fixed $t$, and all derivatives in the parts (a)-(c) are well-defined random variables. It suffices thus to show that each of the derivatives is dominated by an integrable random variable, and the proof follows by an application of the Dominated Convergence Theorem to the difference quotient that defines the derivatives. 

We detail the argument for $\nabla_p \rho^2(p,x(t))$. Note that $\nabla_p \rho^2(p,x(t))=-2\exp^{-1}_p(x(t))$ when $x(t) \notin \mathcal C(x(t))$ with $\|\nabla_p \rho^2(p,x(t))\|_p=2\rho(p,x(t))$.  Then, $\rho(p,x(t))\leq \rho(p,\fre(t))+\rho(\fre(t),x(t))$. Thus, with $\|V^{\gamma}(t)\|_{\gamma (t)}=\|e(t)\|$,
\begin{equation}
\label{eq:DCT}
\|\nabla_p \rho^2(p,x(t))\|_p \leq 2(C_1+\|e(t)\|),
\end{equation}
since $\sup_t\rho(p,\fre(t)) \leq C_1< \infty$ as $t \mapsto \rho(x(t),p)$ is continuous and attains a maximum on $[0,1]$. Note that $e(t)$ has a Gaussian tail, and the right hand side is integrable. Then, for $v \in T_pM$, the limiting difference quotient satisfies
\[
\frac{\text{d}}{\text{d}\epsilon}\Big |_{\epsilon=0} \mathbb E(\rho^2(\exp_{p}(\epsilon v),x(t)))=\mathbb E\left(\frac{\text{d}}{\text{d}\epsilon}\Big |_{\epsilon=0}\rho^2(\exp_{p}(\epsilon v),x(t))\right),
\]
so that $\nabla_p F(p,t)=\mathbb E \left(\nabla_p \rho^2(p,x(t))\right)$.

In operator notation, $\nabla^2_p \rho^2(p,x(t))=2g+2\rho(p,x(t))H_t(p)\rho(p,x(t))$, where
$g$ is the Riemannian metric tensor and $H_t(p):T_pM \to T_pM$ is the (1,1)-tensor that denotes the Hessian operator of $p \mapsto \rho^2(p,x(t))$. From assumption A4 and from the Hessian comparison theorem \citep[Proposition 11.3][]{lee2018introduction}, we get for any $v \in T_pM$
\[
H_t(p)(v,v) \leq h(\sqrt{\kappa} \rho(p,x(t)))\|v\|_p^2,
\]
where $h:\mathbb R \to \mathbb R$
\[
h(r):=
\begin{cases}
 r\coth(r), & \kappa>0   \\
 1 & \kappa=0\thickspace ,
\end{cases}
\]
from which,
\begin{equation}
\label{eq:HessianComp}
\|\nabla^2_p \rho^2(p,x(t))\|_{\text{op}} \leq 2+2\rho(p,x(t)) h(\sqrt{\kappa} \rho(p,x(t))).
\end{equation}
Since $\rho(p,x(t)) \leq C_1+\|z(t)\|$, $\|\nabla^2_p \rho^2(p,x(t))\|_{\text{op}}$ is dominated by an integrable random variable. This proves part (a). 

In part (b), we focus on $t \mapsto \rho^2(p,x(t))$. For the first covariant time derivative, we have $\nabla_tF(p,t)=\langle\nabla_p F(p,t),P^\sigma_{0 \to 1}\dot x(t)\rangle_p$, where $\sigma:[0,1] \to M$ is the geodesic between $\sigma(0)=x(t)$ and $p$, which exists a.e. . Then, from \eqref{eq:DCT}, we obtain
\[
\nabla_tF(p,t) \leq \|\nabla_pF(p,t)\|_p\|\dot x(t)\|_{x(t)} \leq 2(C_1+\|\dot e(t)\|)\|\dot x(t)\|_{x(t)}. 
\]
The time derivative $\dot x(t)$ can be expressed as the value of a Jacobi field $J_t(s)$ along the geodesic $s\mapsto \exp_{\gamma(t)}(sV^{\gamma}(t))$ evaluated at $s=1$. That is, $\dot x(t)=J_t(1)$, where the Jacobi field $J_t(s)$ satisfies the initial conditions
\[
J_t(0)=u, \quad \nabla_s J_t(0)= v,
\]
where $u=\dot \gamma(t)$ and $v=\nabla_tV^{\gamma}(t)$. Decompose $J_t(s)=J^u_t(s)+J_t^v(s)$, where $J^u_t(s)$ is a Jacobi field satisfying $J_t^u(0)=u$ and $\nabla_s J_t^u(0)=0$, and $J^v_t(s)$ is a Jacobi field satisfying $J_t^v(0)=0$ and $\nabla_s J_t^v(0)=v$. Then, $\|\dot x(t)\|_{x(t)} \leq \|J^u_t(1)\|_{\gamma(t)}+|J^v_t(1)\|_{\gamma(t)}$. From the Jacobi comparison theorem \citep[Theorem 11.9][]{lee2018introduction}, and the local isometric property of the rolling map, we get
\begin{align*}
\|J^u_t(1)\|_{\gamma(t)}+ |J^v_t(1)\|_{\gamma(t)}&\leq \cosh(\kappa \|V^{\gamma}(t)\|_{\gamma(t)}\|\dot \gamma(t)\|_{\gamma(t)}\\
&\quad +\frac{1}{\kappa}\sinh(\kappa \|V^{\gamma}(t)\|_{\gamma(t)}\left\|\nabla_tV^{\gamma}(t)\right\|_{\gamma(t)}\\
&=\cosh(\kappa \|e(t)\|\|m(t)\|+\frac{1}{\kappa}\sinh(\kappa \|e(t)\|)\|\dot e(t)\|,
\end{align*}
since the rolling operation is a local isometry. Since $\cosh(y) \geq 1/y \sinh(y)$ for $y>0$ and $\cosh(y) \leq e^{|y|}$,
\[
\|\dot x(t)\|_{x(t)} \leq e^{\kappa \|e(t)\|}\left(\sup_{t \in [0,1]} \|m(t)\|+\|\dot e(t)\|\right). 
\]
Note that for an $\mathbb R^d$-valued Gaussian process $e$, $\|e(t)\|$ has sub-Gaussian tails and thus $\mathbb E(e^{\kappa}\|e(t)\|)<\infty$. Moreover,  $\|\dot e(t)\|$ has moments of all orders since $\dot e$ is another Gaussian process. Thus $\|\dot x(t)\|_{x(t)}$ is dominated by an integrable random variable. 

With $\nabla_p^2F(p,x(t)): T_pM \times T_pM \to \mathbb R$ as the (0,2)-tensor, the second covariant time derivative is
\[
\nabla^2_t F(p,t)=2[\nabla_p^2F(p,x(t))(\dot x(t), \dot x(t))+\langle \nabla_pF(p,t),P^\sigma_{0 \to 1}\dot x(t)\rangle_p], 
\]
with the geodesic $\sigma:[0,1] \to M$ between $p$ and $x(t)$. The argument for the first time derivative ensures that the second term on the right hand is dominated by an integrable random variable, while \eqref{eq:HessianComp} does the same for the Hessian term. 

Integrable upper bounds for the mixed partial derivatives follow from a combination of those for the time and space derivative derived above. We omit the details. 
\end{proof}
The next Lemma develops a uniform concentration bound for an empirical processes indexed by a parametric class of vector-valued functions. We use standard notation from empirical process theory \citep[e.g.][]{van1996weak}.
\begin{lemma}
\label{lem:uniform_conc}
Let $\mathcal X$ be a metric space, and $\mathcal F_\theta:=\{f_\theta:\mathcal X \to \mathbb R^d, \theta \in \Theta\}$ be a class of measurable functions with $\theta \mapsto f_\theta$ uniformly bounded and Lipschitz for a compact parameter space $\Theta$. Let $X,X_1,\ldots,X_n$ be i.i.d. from $\mathbb P$ on $\mathcal X$, and $\sigma^2:=\sup_{f_\theta \in \mathcal F_\theta}\mathbb E (\|f_\theta(X)\|^2)<\infty$. For the empirical process
\[
\mathbb G_n(f_\theta):=(\mathbb P_n-\mathbb P)f_\theta=\frac{1}{n}\sum_{i=1}^n [f_\theta(X_i)-\mathbb E f_\theta(X_i)],
\]
the following hold:
\begin{enumerate}
    \item [(a)] For $0<\epsilon<1$, there exists a constant $C$ such that with probability at least $1-\epsilon$,
    \[
    \sup_{f_\theta \in \mathcal F_\theta}\|\mathbb G_n(f_\theta)\| \leq 
    2 \mathbb E\left(\sup_{f_\theta \in \mathcal F}\|\mathbb G_n(f_\theta)\|\right)+\sqrt{\frac{2\sigma^2 \log(d/\epsilon)}{n}}+\frac{2C\log(d/\epsilon)}{3n}.
    \]
    \item[(b)] There exists a constant $C>0$ such that
    \[
    \mathbb E\left(\sup_{f_\theta \in \mathcal F_\theta}\|\mathbb G_n(f_\theta)\|\right) \leq C \sqrt{\frac{1}{n}}.
    \]
\end{enumerate}
Therefore, 
\[
\sup_{f_\theta \in \mathcal F_\theta}\|\mathbb G_n(f_\theta)\| =O_P(n^{-1/2}).
\]
\end{lemma}
\begin{proof}
Constants in this proofs are generically denoted by $C$. Let 
\[
\mathcal G^u_\theta:=\{x \mapsto \langle f_\theta(x),u \rangle, f_\theta \in \mathcal F, u \in \mathbb R^d\}
\]
be the class of real-valued functions that are one-dimensional projections of functions in $\mathcal F$ along vectors $u$. From Theorem 2.3 of \cite{bousquet2002bennett}, we get that with probability at least $1-\epsilon$
\[
\sup_{g \in \mathcal G^u_\theta}|(\mathbb P_n-\mathbb P)g| \leq \mathbb E \sup_{g \in G^u_\theta}|(\mathbb P_n-\mathbb P)g|+
\sqrt{\frac{2\sigma_g^2 \log(d/\epsilon)}{n}} +\frac{2C\log(d/\epsilon)}{3n},
\]
where $\sigma^2_g:=\sup_{g \in G^u_\theta}\mathbb E (|g(X)|^2)<\infty$, and $C>0$. 
From the scalar symmetrization inequality \citep[Lemma 2.3.1][]{van1996weak} we then have that with probability at $1-\epsilon/d$
\begin{equation}
\label{eq:uc}
\sup_{g \in G^u_\theta}|(\mathbb P_n-\mathbb P)g| \leq 2\mathbb E \left(\sup_{g \in G^u_\theta}\left|\frac{1}{n}\sum_i \eta_i g(X_i)\right|\right)+
\sqrt{\frac{2\sigma_g^2 \log(d/\epsilon)}{n}} +\frac{2C\log(d/\epsilon)}{3n},
\end{equation}
where $\eta_i$ are i.i.d. Rademacher random variables. Then, the probability that \eqref{eq:uc}  holds for all coordinates, via a bound over the union of each coordinates, is at least $1-\epsilon$. Taking supremum over vectors $u \in \mathbb R^d$ on both sides, and moving the supremum within the expectation on the right-hand side results in
\[
\sup_u\sup_{g \in G^u_\theta}|(\mathbb P_n-\mathbb P)g| \leq 2 \mathbb E \left(\sup_u \sup_{g \in G^u_\theta}\left|\frac{1}{n}\sum_i \eta_i g(X_i)\right|\right)+
\sqrt{\frac{2\sigma^2 \log(d/\epsilon)}{n}} +\frac{2C\log(d/\epsilon)}{3n},
\]
since $\sigma^2_g \leq \sigma^2$. The left-hand side is lower bounded by $  \sup_{f_\theta \in \mathcal F}\|\mathbb G_n(f_\theta)\|$ when $u$ is restricted to lie on the unit sphere in $\mathbb R^d$. Use $\|u\| \leq \sqrt{d} \max_j u_j$ to upper bound the double supremum over the Rademacher sum on the right-hand side by $\mathbb E\left(\sup_{f_\theta \in \mathcal F_\theta}\|\mathbb G_n(f_\theta)\|\right)$ along with terms that introduce at most a $\sqrt{d}$ factor. This proves (a). 

From Dudley's entropy integral bound \citep{van1996weak}, we have,
\begin{align*}
\mathbb E\left(\sup_{f_\theta \in \mathcal F_\theta}\|\mathbb G_n(f_\theta)\|\right) &\leq \frac{C}{\sqrt{n}}\int_{0}^\sigma \sqrt{\log N(\mathcal F, \tilde d, \epsilon)} \text{d}\epsilon\\
\end{align*}
where $N(\mathcal F, \tilde d, \epsilon)$ is the covering number of $\mathcal F_\theta$ with respect to the metric $\tilde d(f_\theta,g_\theta)^2:=\mathbb E\|f_\theta(X)-g_\theta(X)\|^2$. Since functions in $\mathcal F_\theta$ are Lipschitz in $\theta$, $\log N(\mathcal F_\theta, \tilde d, \epsilon) \lesssim d\log(C/\epsilon)$, where $C$ depends on the diameter of the parameter space $\Theta$ and Lipschitz constant of $f_\theta$ \citep[e.g.][p. 271]{van2000asymptotic}. The integral
\[
\int_{0}^{\sigma} \sqrt{d \log (C/\epsilon)}\text{d}\epsilon
\]
is bounded; this is seen by substituting $u^2=d \log (C/\epsilon)$ and using integration by parts. This proves (b). 
\end{proof}

\begin{lemma}
\label{lem:small_prob}
Assume $m^\uparrow=\fre$. Under assumptions A1-A5, if the sample paths of a rolled Gaussian process $x$ are continuous a.s., then $\sup_{t \in [0,1]}\rho(x(t),\hatfre(t))=O_P(1)$. 
\end{lemma}
\begin{proof}
From Proposition \ref{prop:vector_field},
\[
t \mapsto V^{\fre}(t):=P^{\fre}_{0 \to t}P^c_{0 \to 1}(y_b(t)-m_b(t)), \quad t \in [0,1],
\]
is a Gaussian vector field along $\gamma_{\text{Fre}}$, such that  $\|V^{\fre}(t)\|_{\fre(t)}=\|y_b(t)-m_b(t)\|_b=\|z(t)-m(t)\|$, since parallel transport maps and the frame $U$ are isometries. Let $z(t)=m(t)+e(t)$, where $e$ is a zero-mean Gaussian process with covariance $K$. From the 

With $r_0>0$ as the radius of the tubular neighbourhood $\mathbb T_{\gamma_{\text{Fre}}}(r_0)$ in \eqref{eq:tubular}, under assumption A4, choose $\tau$ such that $0<\beta<\tau <r_0/2$, where $\beta:=E(\|e(t)\|) < \infty$, from the Borell-TIS inequality \citep[e.g.][Chapter 2]{adler2007random}. Let $A_\tau:=\{\sup_{t \in [0,1]}\|e(t)\| \leq \tau \}$. Again, from the Borell-TIS inequality, 
\[
\mathbb P(A^c_\tau)\leq \exp{\left\{-\frac{(\tau -\beta)^2}{2\sigma^2}\right\}},
\]
where $\sigma^2$ is the uniform bound on the covariance of $z$ in assumption A4. 

Define $B_n:=\{\sup_{t \in [0,1]} \rho(\hatfre(t), \fre(t)) \leq r_0/2\}$. From Theorem \ref{thm:C1_convergence} $\mathbb P(B_n) \to 1, n \to \infty$. On the set $A_\tau \cap B_n$, all points $\hatfre(t)$ and $x(t)$ are contained within a subset of $M$ within which the inverse exponential map is well-defined. Thus, on $A_\tau \cap B_n$, 
\begin{align*}
\rho(\hatfre(t), x(t)) &\leq \rho(\hatfre(t), \fre(t))+\rho(\fre(t), x(t))\\
& \leq \frac{r_0}{2}+\tau \leq \text{inj}(\hatfre([0,1])),
\end{align*}
since $\rho(x(t),\fre(t))=\|e(t)\| \leq \tau$ upon noting that $[0,1] \ni t \mapsto \exp_p(tv)$ is minimizing for $\|v\|_p < \text{inj}(p)$, and $\rho(\hatfre(t),\fre(t)) \leq r_0/2$. Thus $\|\exp^{-1}_{\hatfre(t)}(x(t))\|_{\hatfre(t)} \leq \tau+\frac{r_0}{2}$ for every $t$, and 
\[
\sup_{t \in [0,1]}\|\exp^{-1}_{\hatfre(t)}(x(t))\|_{\hatfre(t)} \leq \tau+\frac{r_0}{2}.
\]
Hence for every $\epsilon>0$, we can choose $\tau$ and $n$ large enough so that $\mathbb P(A^c_\tau) \leq \epsilon/2$ and $\mathbb P(B^c_n) \leq \epsilon/2$, which ensures that
\[
\mathbb P\left(\sup_{t \in [0,1]}\|\exp^{-1}_{\hatfre(t)}(x(t))\|_{\hatfre(t)} \leq \tau+\frac{r_0}{2}\right) \geq 1-\epsilon. 
\]
\end{proof}

\section{Proofs of main results}
\label{sec:proofs}
\subsection{Proof of Theorem \ref{thm:frechet:mean}}

 For every $t$, suppose $F(y):=E\{\rho^2(y,x(t))\} < \infty$ at a point $y$, and hence on all $M$. For part $(i)$, when every point $p \in M$ has empty cut locus the function $y \mapsto F(y)$ is geodesically convex on $M$ with a global minimizer that is the unique stationary point, and the Fr\'echet mean of $x(t)$ is hence unique \citep[e.g.][]{afsari2011riemannian}. On such $M$, $\exp_{\gamma(t)}:T_{\gamma(t)}M \to M$ is bijective, and it is thus sufficient to show that
$E\{\exp^{-1}_{\gamma(t)}(x(t))\}=0$ \citep[e.g.][]{le2001locating}.
Since 
$
\exp^{-1}_{\gamma(t)}(x(t))=  P^\gamma_{t \leftarrow 0}P^c_{1 \leftarrow 0}\{y(t)-m_b(t)\},
$
and the parallel transports are linear isometries, the result follows since $E\{y(t)-m_b(t)\}=0$. 

For part $(ii)$ note that the function $F$ defined above may be written as
\begin{align*}
F(y)&=\int_M \rho^2(y,x(t)) \text{d}(\exp_{\gamma(t)})_\#\lambda(x(t))=\int_{M}\rho^2(\sigma_{\gamma(t)}(y),\sigma_{\gamma(t)}(x(t))) \text{d}(\exp_{{\gamma(t)}})_\#\lambda(x(t)),
\end{align*}
since $M$ is symmetric and $\sigma_{\gamma(t)}$ is an isometry on $M$. Let $v \in T_{\gamma(t)}M$ be such that $\exp_{\gamma(t)}(v)=x(t)$. The distribution of $x(t)$ is of the form $(\exp_{\gamma(t)})_\# \lambda$ of a distribution $\lambda$ with Lebesgue-density $f$ that is symmetric around the origin in $T_{\gamma(t)}M$ such that $f(v)=f(-v)$; this holds regardless of whether $v$ is unique when $x(t) \notin \mathcal C(\gamma(t))$ or a measurable selection is made from the set of all $v$ such that $\exp_{\gamma(t)}(v)=x(t)$. As a consequence, the distribution of $x(t)$ possesses geodesic symmetry. Since the isometry $\sigma_{\gamma(t)}$ reverses geodesics through ${\gamma(t)}$, we have the relation $\exp_{\gamma(t)}(-v)=\sigma_{\gamma(t)}(\exp_{\gamma(t)}(v))$. Therefore, 
\begin{align*}
F(y)&=\int_{M}\rho^2(\sigma_{\gamma(t)}(y),\sigma_{\gamma(t)}(x(t))) \text{d}(\exp_{{\gamma(t)}})_\#\lambda(x(t))\\
&=\int_{M}\rho^2(\sigma_{\gamma(t)}(y),\sigma_{\gamma(t)}(x(t))) \text{d}(\exp_{{\gamma(t)}})_\#\lambda(\sigma_{\gamma(t)}(x(t)))\\
&=F(\sigma_{\gamma(t)}(y)),
\end{align*}
and $F:M \to \mathbb R$ is symmetric about $\gamma(t)$.  

We now establish that is suffices to show that $\gamma(t)$ is the stationary point of $F$ in order for it to be the Fr\'echet mean. Assume otherwise, and let $\gamma(t)$ be a stationary point of $F$ whose unique minimizer is $q$ distinct from $\gamma(t)$. Since $F$ is symmetric about $\gamma(t)$, we have then that $\sigma_{\gamma(t)}(q)$ must also minimizer $F$, which contradicts the uniqueness assumption of the Fr\'echet mean, unless $q=\sigma_{\gamma(t)}(q)$. This cannot happen since by assumption $q \notin \mathcal C(\gamma(t))$. The reasoning from proof of $(i)$ can be reused to show that $\gamma(t)$ is a stationary point of $F$, and complete the proof. 




\subsection{Proof of Proposition \ref{prop:vector_field}}
$E^\gamma=\{e_1,\ldots,e_d\}$ be a frame along $\gamma$, defined in Appendix \ref{app:ortho_frame}. For every $t$, note that $y_b(t)-m_b(t) \sim N_d(0,K_b(t,t))$. For every $t$, the orthonormal frame $E^\gamma(t): T_{\gamma(0)}M \to T_{\gamma(t)}M$ is a linear operator. Consequently,  
\[
V^\gamma(t) \sim N_d\left(0, E^\gamma(t) K_b(t,t)E^\gamma(t)^\top\right)
\]
 with values in the vector space $(T_{\gamma(t)}M,\langle \cdot, \cdot \rangle_{\gamma(t)})$, where $E^\gamma(t)^\top$ is the corresponding adjoint operator.  Consider for distinct times $t,t'$ the random vector $(V^\gamma(t), V^\gamma(t'))$ with values in the direct sum inner product space $\left(T_{\gamma(t)}M \oplus T_{\gamma(t')}M{,}\langle \cdot, \cdot \rangle_{\gamma(t)}+\langle \cdot, \cdot \rangle_{\gamma(t')}\right)$. Then, $(V^\gamma(t), V^\gamma(t'))$ is a zero-mean Gaussian with covariance
 \begin{equation*}
 \begin{pmatrix}
 \Omega_{tt} &  \Omega_{tt'}\\
 \Omega_{tt'}^\top & \Omega_{t't'}
 \end{pmatrix},
 \end{equation*}
  where $\Omega_{ab}:=E^\gamma(a) K_b(a,b)E^\gamma(b)^\top $ (see Appendix \ref{sec:Gaussian})
The discussion above holds for any curve $\gamma$, and thus for $\gamma=m_b^\uparrow$. 
 
 It is now straightforward to extend the above to the case of $r$ times $t_1,\ldots,t_r$ involving a direct sum of $r$ tangent spaces and corresponding $r \times r$ block covariance operator comprising $r \choose 2$ linear operators $\Omega_{t_i t_j}$ for $i,j=1,\ldots,r$. This establishes the existence of finite-dimensional marginal Gaussian distributions corresponding to the choices $m$ and $k$. If the family of finite-dimensional marginal distributions constitutes a \emph{projective} family, existence and uniqueness of a Gaussian vector field $t \mapsto V^\gamma(t)$ along $\gamma$ follows from Kolmogorov's extension theorem \citep[e.g.,][]{bhattacharya2007basic}. Proof of projectivity follows straightforwardly from Proposition 16 in \cite{hutchinson2021vector}, who consider general Gaussian vector fields not necessarily along curves $\gamma$.

\subsection{Proof of Proposition \ref{prop:equivariance}}
\label{sec:proof:of:equivariance}
Given a mean $m$ and the covariance $K$ for the $\mathbb R^d$-valued Gaussian process $z$, the pair $(b,U)$ determines $m_b=Um$ and $K_b=UKU^\top$ in $T_bM$, the mean and kernel, respectively, of the Gaussian process $y_b(t)$ in $T_bM$; the unique rolling of $m_b$ onto $M$ then results in $m_b^\uparrow$. Upon setting $\gamma=m_b^\uparrow$, the Gaussian vector field $V^\gamma(t)$ is obtained by action of the frame $E^\gamma(t)$, determined by $E^\gamma(0)=(P^c_{1 \leftarrow 0}u_1,\ldots,P^c_{1 \leftarrow 0}u_d)$, on the Gaussian process $y_b(t)$; in other words, the map $(U,m,K) \to V^{\gamma}$ is injective. The curve $x$ on $M$ is then defined as $x(t)=\exp_{\gamma(t)}(V^\gamma(t))$.

Choose another point $b' \in M$ with basis $U'$ for $T_{b'}M$, resulting in the process $x'$ on $M$. From the discussion above, $x'$ will have the same distribution as $x$ if $m'$ and $k'$ can be chosen such that the parallel transport of $U'm'$ and $U'K'U^{'\top}$ to $T_bM$ coincides with $m_b=Um$ and $K_b=UKU^\top$. The isometric parallel transport map $P^{\eta}_{1 \leftarrow 0}: T_{b'}M \to T_bM$ along a curve $\eta:[0,1] \to M$ connecting $b'$ and $b$ ensures that $P^{\eta}_{1 \leftarrow 0}U'$ is an orthonormal basis on $T_bM$. Thus choosing $m':=(P^{\eta}_{1 \leftarrow 0}U')^{-1}Um=U^{'\top} P^{\eta}_{0 \leftarrow 1}Um$ in $\mathbb R^d$  gives us the desired result; a similar argument applies for $K'$ as well. The chosen $m'$ and $K'$ are unique, since the parallel vector fields along $\eta$ determined by $U$ and $U'$ that generate the parallel transport maps, respectively,  from $T_{b'}M$ to $T_bM$ and vice versa are unique. If $\eta$ is chosen to be the piecewise curve consisting of a segment connecting $b'$ to $\gamma(0)$ then a segment connecting $\gamma(0)$ to $b$ then the result follows.

\subsection{Proof of Theorem \ref{thm:C1_convergence}}
\label{app:proof_C1}
From Lemma \ref{lem:tubular}, the tubular neighbourhood $\mathbb T_{\gamma_{\text{Fre}}}(r_0)$ in \eqref{eq:tubular} is compact with $r_0>$ and assumption A4 implies that $|\sec{(\mathbb T_{\gamma_{\text{Fre}}}(r_0))}| \leq \kappa$ for some $\kappa \geq 0$. For $p,q \in \mathbb T_{\gamma_{\text{Fre}}}(r_0)$, by Lemma \ref{lemma:Lipschitz}, the inverse exponential is  smooth in both $p$ and $q$, so that
\[
\|\exp^{-1}_p(q)\|_p \leq C_0, \quad \|\nabla_p\exp^{-1}_p(q)\|_p \leq C_1, \quad \|\nabla^2_p\exp^{-1}_p(q)\|_p \leq C_2, 
\]
uniformly over $p$ and $q$ with constants that depend only on the geometry of $M$ via $\kappa$. Consider
\[
t \mapsto V^{\fre}(t):=P^{\fre}_{0 \to t}P^c_{0 \to 1}(y_b(t)-m_b(t)), \quad t \in [0,1],
\]
the Gaussian vector field along $\fre$ from Proposition \ref{prop:vector_field}. By assumption A3, we note that $\sup_t\mathbb E\|V^{\fre}(t)\|_{\fre(t)}<\infty $, and the operator norm of the covariance $\Sigma(s,t):T_{\fre(s)}M \to T_{\fre(t)}M$, $\Sigma(s,t)$ is uniformly (in $s$ and $t$) bounded by the constant $\sigma^2$, since the frame $U$ and parallel transport maps are isometries; this can be seen by expressing $\Sigma(s,t)$ in local coordinates as in the proof of Proposition \ref{prop:vector_field}.     

\subsubsection{Proof of (a)}
Let $S:= \mathbb T_{\gamma_{\text{Fre}}}(r_0) \times [0,1] \to \mathbb R$. Consider the map $\Phi:S \to \mathbb R$
\[
\Phi(p,t)=\mathbb E [\exp^{-1}_p(x(t))].
\]
From Lemma \ref{lemma:Lipschitz}, note that $(p,q) \mapsto \exp^{-1}_p(q)$ is $C^2$ with bounded derivatives within $ S$. Then, due to Lemma \ref{lemma:FrechetFunction_deriv}, the map $\Phi$ is also $C^2$ on $S$. 
 
The map $\Phi$ is the restriction of $(p,t)\mapsto -\frac{1}{2}\nabla_p F(p,t)$ to $S$. Then, $\nabla_p \Phi(p,t)=-\frac{1}{2}\nabla^2_p F(p,t)=-\frac{1}{2}H_t$, the Hessian of $F$. By assumption A1, $\fre(t)$ for each $t$ is a strict local minimum of $p\mapsto F(p,t)$. And, by assumption A5, for each $t$, the (1,1)-tensor $H_t: T_{\fre(t)}M \to T_{\fre(t)}M$ is positive definite and invertible. 

Fix $t_0 \in [0,1]$. By the implicit function theorem, there exists an open interval $I_{t_0} \in [0,1]$ and a unique $C^2$ curve $\phi_{t_0}:I_{t_0} \to M$, such that for all $t \in I_{t_0}$
\[
\Phi(\phi_{t_0}(t),t)=0, \quad \phi_{t_0}(t_0)=\fre(t_0);
\]
see Appendix \ref{app:c_k curves} for a definition of $C^2$ curves using local coordinates. 
 Choose a finite subcover $\{I_{t_1},\ldots,I_{t_l}\}$ of $[0,1]$ (due to its compactness) so that on each $I_{t_j}$, $\fre$ equals a $C^2$ curve $\phi_{t_j}$, and on overlaps $I_{t_j} \cap I_{t_k}$ the functions $\phi_{t_j}$ and $\phi_{t_k}$ must agree; this is so, since both solve $\Phi(p,t)=0$, coincide at some point, and thus everywhere by uniqueness. Thus $\fre$ is $C^2$ on the subcover $\{I_{t_1},\ldots,I_{t_l}\}$, and hence $C^2$ on $[0,1]$.

\subsubsection{Proof of (b)}
Consider the empirical version
\[
(p,t) \mapsto \Phi_{n}(p,t):=\frac{1}{n} \sum_{i=1}^n \exp^{-1}_p(x_i(t)),
\]
of $\Phi$, 
which is $C^2$ on $S$, and satisfies
\[
\Phi_{n}(\hatfre(t), t)=0 \text{ a.s.}
\]
From the arguments in part (a) above, it follows that $\Phi_{n}$ is also $C^2$ on $S$.  If
\begin{align*}
\sup_{(p,t) \in S}\|\Phi_{n}(p,t)-\Phi(p,t)\|_p &\to 0 \text{ a.s.}\\
\sup_{(p,t) \in S}\|\nabla_p\Phi_{n}(p,t)-\nabla_p\Phi(p,t)\|_p &\to 0 \text{ a.s.},
\end{align*}
for large enough $n$, using arguments from part (a) and the implicit function theorem, $\hatfre \in \mathbb T_{\gamma_{\text{Fre}}}(r_0)$ a.s and is $C^2$. It thus suffices to show that the parametric function classes
\begin{align*}
\mathcal F&:=\{f_\theta(x):=\exp_p^{-1}(x(t)), \theta:=(p,t) \in S \}\\
\mathcal G&:=\{g_\theta(x):=\nabla_p\exp_p^{-1}(x(t)), \theta:=(p,t) \in S\}
\end{align*}
are each Glivenko-Cantelli \cite{van1996weak}. Note that $S$ is compact, and from the properties of the inverse exponential in Lemma \ref{lemma:Lipschitz}, the claim readily follows, since both $\mathcal F$  and $\mathcal G$ are equicontinuous and uniformly bounded. See Lemma \ref{lem:uniform_conc} for a precise rate of convergence. 

\subsubsection{Proof of (c)}
From Lemma \ref{lem:uniform_conc}, the parametric function classes $\mathcal F$ and $\mathcal G$ in the proof of (b) are in fact Donsker. Thus, 
\begin{align*}
\sup_{(p,t) \in S}\|\Phi_{n}(p,t)-\Phi(p,t)\|_p &=O_P(n^{-1/2})\\
\sup_{(p,t) \in S}\|\nabla_p\Phi_{n}(p,t)-\nabla_p\Phi(p,t)\|_p &=O_P(n^{-1/2}).
\end{align*}
For a fixed $t$, let $v_n(t):=\exp^{-1}_{\fre(t)}(\hatfre(t))$. A Taylor expansion of $p \mapsto \Phi_{n}(p,t)$ gives
\[
0=\Phi_{n}(\hatfre(t),t)=\Phi_{n}(\fre(t),t)+\nabla_{\fre(t)}\Phi_{n}(\fre(t),t)(v_n(t))+R_n(t),
\]
where $\|R_n(t)\|_{\fre(t)} \leq C \|v_n(t)\|_{\fre(t)}$, since from Lemma \ref{lemma:Lipschitz}, $p \mapsto \nabla_p\Phi_{n}(p,t)$ is uniformly bounded within $\mathbb T_{\gamma_{\text{Fre}}}(r_0)$. Subtracting $\Phi(\fre(t),t)=0$ and re-arranging gives
\begin{align*}
H_tv_n(t)=&-(\Phi_{n}(\fre(t),t)-\Phi(\fre(t),t))\\
&-(\nabla_{\fre(t)}\Phi_{n}(\fre(t),t)-\nabla_{\fre(t)}\Phi(\fre(t),t))v_n(t)\\
&-R_n(t), 
\end{align*}
where $H_t$ is the Hessian operator from part (a). From assumption A5, $\|H_t^{-1}\|_{\text{op}} \leq 1/\lambda_{\min}$, and when combined with the fact that $\mathcal F$ and $\mathcal G$ are Donsker, we get
\[
\|v_n(t)\|_{\fre(t)} \leq \frac{1}{\lambda_{\min}}\left(C_1n^{-1/2}+C_2n^{-1/2}\|v_n(t)\|_{\fre(t)} +C_3\|v_n(t)\|_{\fre(t)} ^2\right)
\]
uniformly in $t$. Re-arranging, 
\[
\left(1-1/\lambda_{\min}C_2n^{-1/2}\right)\|v_n(t)\|_{\fre(t)} \leq \frac{1}{\lambda_{\min}}\left(C_1n^{-1/2}+C_3\|v_n(t)\|^2_{\fre(t)}\right).
\]
For large $n$, the term $(1-1/\lambda_{\min}C_2n^{-1/2})$ is bounded below by 1/2. Thus,
\[
\|v_n(t)\|_{\fre(t)} \leq \frac{2}{\lambda_{\min}}\left(C_1n^{-1/2}+C_3\|v_n(t)\|_{\fre(t)} ^2\right),
\]
which determines a fixed point inequality $bu^2-u+an^{-1/2} \geq 0$ in the variable $u$ for suitable constants $a, b$. The roots when equating the left-hand side to zero are
\[
u=\frac{1}{2b}(1\pm \sqrt{1-4abn^{-1/2}}),
\]
which for large $n$ gives $u \approx an^{-1/2}$ so that any $u$ satisfying the fixed-point inequality must satisfy $u \leq 2an^{-1/2}$. As a consequence, $\|v_n(t)\|_{\fre(t)}=O_P(n^{-1/2})$. Noting that $\|v_n(t)\|_{\fre(t)}=\rho(\hatfre(t),\fre(t))$ gives
\[
\sup_{t \in [0,1]}\rho(\hatfre(t),\fre(t))=O_P(n^{-1/2}).
\]

Next, differentiating the identity $\Phi_{n}(\hatfre(t),t)=0$ covariantly with respect to $t$
\[
\nabla_{\fre(t)}\Phi_{n}(\hatfre(t),t))(\nabla_t \hatfre(t))+\nabla_t\Phi_{n}(\hatfre(t),t))=0;
\]
in similar fashion, for the identity $\Phi(\fre(t),t)=0$, we get
\[
\nabla_{\fre(t)}\Phi(\fre(t),t))(\nabla_t \fre(t))+\nabla_t\Phi(\fre(t),t))=0.
\]
Subtracting the two we obtain,
\begin{align*}
H_t\big(\nabla_t\hatfre(t) - \nabla_t\fre(t)\big)
 &= -[\nabla_{\fre(t)}\Phi_{n} - \nabla_{\fre(t)}\Phi](\hatfre(t),t)
      (\nabla_t\hatfre(t)) \\
 &\quad -[\nabla_t\Phi_{n} - \nabla_t\Phi](\hatfre(t),t) \\
 &\quad -[\nabla_{\fre(t)}\Phi(\hatfre(t),t)-H_t]
      (\nabla_t\hatfre(t)) \\
 &\quad -[\nabla_t\Phi(\hatfre(t),t)-\nabla_t\Phi(\fre(t),t)].
\end{align*}
Differences between tangent vectors in $T_{\fre(t)}M$ and $T_{\hatfre(t)}M$ in the above expression is to be understood either in normal coordinates of an arbitrary point within $\mathbb T_{\gamma_{\text{Fre}}}(r_0)$ or by identifying the tangent spaces by parallel transporting vectors along the unique geodesic $\sigma_t:[0,1] \to \mathbb T_{\gamma_{\text{Fre}}}(r_0) \subset M$ from $\hatfre(t)$ to $\fre(t)$. 

The first term \(O_p(n^{-1/2})\) since the function class $\mathcal G$ defined in the proof of (b) is Donsker. The function class $\{h_\theta(x):=\nabla_t\exp_p^{-1}(x(t)), \theta \in S \}$ is also Donsker owing to the uniform boundedness and equicontinuity of the class, and the second term is \(O_p(n^{-1/2})\). The last two terms are \(O_p(n^{-1/2})\)
due to the \(C^2\)-smoothness of \(\Phi\)
combined with \(\sup_t \rho(\hatfre(t),\fre(t)) = O_p(n^{-1/2})\). Since  \(\|H_t^{-1}\|\le\lambda_{\min}^{-1}\),
multiplying both sides by \(H_t^{-1}\) yields
\[
  \sup_{t\in[0,1]}
  \|P_{0 \to 1}^{\sigma_t}\nabla_t\widehat\gamma_n(t)-\nabla_t\gamma(t)\|
  = O_p(n^{-1/2}).
\]
We have thus established that
\[
\sup_{t \in [0,1]}\rho(\hatfre(t),\fre(t))+\|P_{0 \to 1}^{\sigma_t}\nabla_t\widehat\gamma_n(t)-\nabla_t\gamma(t)\|
  = O_p(n^{-1/2}).
\]
This completes the proof. 

\subsection{Proof of Theorem \ref{thm:rates}}
Claims in (a) and (b) are proved if it can be shown that for any $i=1,\ldots,n$,
\[
\sup_{t \in [0,1]}\|x^{\downarrow\hat \gamma_{\text{Fre}}}_{b,i}(t)-x^{\downarrow
\gamma_{\text{Fre}}}_{b,i}(t)\|_b=O_P(n^{-1/2}).
\]
Recall from the discussion following assumption A1 that both $\exp^{-1}_{\hatfre(t)}(x_i(t))$ and $\exp^{-1}_{\fre(t)}(x_i(t))$ are well-defined with probability one. In what follows, for simplicity, we suppress the subscript $i$ in $x^{\downarrow\hat \gamma_{\text{Fre}}}_{b,i}$ and in $x_i(t)$.

Denote by $\sigma_t:[0,1] \to M$ the minimizing geodesic from $\sigma_t(0)=\hatfre(t)$ to $\sigma_t(1)=\fre(0)$. Then, 
\begin{align*}
x^{\downarrow\hat \gamma_{\text{Fre}}}_{b}(t)-x^{\downarrow
\gamma_{\text{Fre}}}_{b}(t)
&=P^{\hat c}_{0 \leftarrow 1}P^{\hatfre}_{0 \leftarrow t}\exp^{-1}_{\hatfre(t)}x(t)-P^{c}_{0 \leftarrow 1}P^{\sigma_t}_{0 \to 1}\exp^{-1}_{\hatfre(t)}x(t) \\
&\quad + P^{c}_{0 \leftarrow 1}P^{\sigma_t}_{0 \to 1}\exp^{-1}_{\hatfre(t)}x(t)-P^{c}_{0 \leftarrow 1}P^{\fre}_{0 \leftarrow t}\exp^{-1}_{\fre(t)}x(t)\\
&=:E(t)\exp^{-1}_{\hatfre(t)}x(t)+B(t),
\end{align*}
where 
\begin{align*}
E(t)&:=P^{\hat c}_{0 \leftarrow 1}P^{\hatfre}_{0 \leftarrow t}-P^{c}_{0 \leftarrow 1}P^{\sigma_t}_{0 \to 1},\\
B(t)&:=P^{c}_{0 \leftarrow 1}\left(P^{\sigma_t}_{0 \to 1}\exp^{-1}_{\hatfre(t)}x(t)-P^{\fre}_{0 \leftarrow t}\exp^{-1}_{\fre(t)}x(t)\right).
\end{align*}
Then, 
\[
\sup_{t \in [0,1]}\|x^{\downarrow\hat \gamma_{\text{Fre}}}_{b}(t)-x^{\downarrow
\gamma_{\text{Fre}}}_{b}(t)\|_b \leq \sup_{t \in [0,1]}\|E(t)\|_{\text{op}}\sup_{t \in [0,1]}\|\exp^{-1}_{\hatfre(t)}x(t)\|_{\hatfre(t)}+\sup_{t \in [0,1]}\|B(t)\|_{\text{op}},
\]
and since $\sup_{t \in [0,1]}\|\exp^{-1}_{\hatfre(t)}x(t)\|_{\hatfre(t)}=O_P(1)$ by Lemma \ref{lem:small_prob}, it suffices to prove that the operator bounds on $E(t)$ and $B(t)$ are both $O_P(n^{-1/2})$ uniformly in $t$. 

Consider first $E(t)$. Adding and subtracting $P^c_{0 \leftarrow 1}P^{\sigma_0}_{0 \to 1}P^{\hatfre}_{0 \leftarrow t}$, we obtain
\[
E(t)=P^{c}_{0 \leftarrow 1}\left(P^{\sigma_0}_{0 \to 1}P^{\hatfre}_{0 \leftarrow t}-P^{\sigma_t}_{0 \to 1}\right)+\left(P^{\hat c}_{0 \leftarrow 1}-P^c_{0 \leftarrow 1}P^{\sigma_0}_{0 \to 1}\right)P^{\hatfre}_{0 \leftarrow t},
\]
so that
\begin{equation}
\label{eq:tri}
\|E(t)\|_{\text{op}} \leq \|P^{\sigma_0}_{0 \to 1}P^{\hatfre}_{0 \leftarrow t}-P^{\sigma_t}_{0 \to 1}\|_{\text{op}}+\|P^{\hat c}_{0 \leftarrow 1}-P^c_{0 \leftarrow 1}P^{\sigma_0}_{0 \to 1}\|_{\text{op}}.
\end{equation}

Note that $P^{\sigma_0}_{0 \to 1}P^{\hatfre}_{0 \leftarrow t}-P^{\sigma_t}_{0 \to 1}$ is the difference operator that encodes holonomy around the geodesic triangle with vertices $\fre(0), \hatfre(0)$ and $\hatfre(t)$. Similarly, $P^{\hat c}_{0 \leftarrow 1}-P^c_{0 \leftarrow 1}P^{\sigma_0}_{0 \to 1}$ is the difference operator pertaining to the triangle formed by $b, \hatfre(0), \fre(0)$. From Lemma \ref{lemma:holonomy}, the right hand side of \eqref{eq:tri} is upper bounded by 
\[
C\kappa\Big([\rho(\hatfre(0),\fre(0))\rho(\hatfre(0),b)][\rho(\hatfre(0),\fre(0))\rho(\hatfre(0),\hatfre(t))]\Big)+o_P(A_1)+o_P(A_2),
\] 
where $A_1$ and $A_2$ are areas of the two geodesic triangles. From Theorem \ref{thm:C1_convergence} (c), $\rho(\hatfre(0),\fre(0))=O_P(n^{-1/2})$ and $\rho(\hatfre(0),\hatfre(t))=O_P(1)$. The point $b$ may be chosen so that the distance $\rho(\hatfre(0),b)$ is bounded in probability; indeed if $b \in \mathbb T_{\gamma_{\text{Fre}}}(r_0)$ then $\rho(\hatfre(0),b) \leq r_0$ almost surely, since from \ref{thm:C1_convergence} (b), $\hatfre \in \mathbb T_{\gamma_{\text{Fre}}}(r_0)$ almost surely. The areas $A_1$ and $A_2$ are upper bounded by the product of the lengths of any two of the sides of the triangles. Constants $C$ and $\kappa$ do not depend on $t$, and consequently, 
\[
\sup_{t \in [0,1]}\|E(t)\|_{\text{op}}=O_P(n^{-1/2}),
\]
which upon combining with the result of Lemma \ref{lem:small_prob} proves that
\[
\sup_{t \in [0,1]}\|E(t)\|_{\text{op}}\|\exp^{-1}_{\hatfre(t)}x(t)\|_{\fre(t)}=O_P(n^{-1/2}). 
\]

We now move to $B(t)$. Since $\|B(t)\|_{\text{op}}$ is preserved under the isometry $P^c_{0 \leftarrow 1}$, we focus on $P^{\sigma_t}_{0 \to 1}\exp^{-1}_{\hatfre(t)}x(t)-P^{\fre}_{0 \leftarrow t}\exp^{-1}_{\fre(t)}x(t)$. Adding and subtracting $P^{\fre}_{0 \leftarrow t}P^{\beta_t}_{0 \to 1}\exp^{-1}_{\hatfre(t)}x(t)$, where $\beta_t:[0,1] \to M$ is the minimizing geodesic from $\beta_t(0)=\hatfre(t)$ to $\beta_t(1)=\fre(t)$, we obtain
\begin{align*}
&\left(P^{\sigma_t}_{0 \to 1} - P^{\fre}_{0 \leftarrow t}P^{\beta_t}_{0 \to 1}\right)\exp^{-1}_{\hatfre(t)}x(t)+P^{\fre}_{0 \leftarrow t}\left(P^{\beta_t}_{0 \to 1}\exp^{-1}_{\hatfre(t)}x(t)-\exp^{-1}_{\fre(t)}x(t)\right)\\
&=\Lambda_1(t)\exp^{-1}_{\hatfre(t)}x(t)+P^{\fre}_{0 \leftarrow t}\Lambda_2(t),
\end{align*}
where $\Lambda_1(t):=P^{\sigma_t}_{0 \to 1} - P^{\fre}_{0 \leftarrow t}P^{\beta_t}_{0 \to 1}:T_{\hatfre(t)}M \to T_{\fre(0)}M$ and $\Lambda_2(t):=P^{\beta_t}_{0 \to 1}\exp^{-1}_{\hatfre(t)}x(t)-\exp^{-1}_{\fre(t)}x(t) \in T_{\fre(t)}M$. Since $P^{\fre}_{0 \leftarrow t}$ is an isometry, and by Lemma \ref{lem:small_prob} since $\exp^{-1}_{\hatfre(t)}x(t)=O_P(1)$, it suffices to show that $\|\Lambda_1(t)\|_{\text{op}}$ and $\|\Lambda_2(t)\|_{\fre(t)}$ are both $O_P(n^{-1/2})$, uniformly in $t$. 

The operator $\Lambda_1(t)$ is a difference operator related to the holonomy of the geodesic triangle formed by $\hatfre(t), \fre(t), \fre(0)$. From Lemma \ref{lemma:holonomy}, repeating the argument above when bounding $E(t)$, we get 
\[
\sup_{t \in [0,1]}\|\Lambda_1(t)\|_{\text{op}}=O_P(n^{-1/2}). 
\]
To bound $\Lambda_2(t)$, we restrict attention to the set $A_\tau \cap B_n$ defined in the proof of Lemma \ref{lem:small_prob} for  $\tau<r_0/2$, where $r_0$ is the radius of the tubular region $\mathbb T_{\gamma_{\text{Fre}}}(r_0)$. As in the proof of Lemma \ref{lem:small_prob}, under assumption A5, choose $0<\zeta<\tau<r_0/2$ where $\zeta:=\sup_{t \in [0,1]}\|m(t)\|$ so that $\mathbb P(A^c_\tau) \leq \exp{\left(-\frac{(\tau-\zeta)^2}{2 \sigma^2}\right)}$, and $\mathbb P(B_n) \to 1$ as $n \to \infty$. Thus, on the high probability set $A_\tau \cap B_n$, pairs realisations of the tuple $(x,\hatfre)$ lie within the compact set
\[
S_\tau:=\left\{(p,q) \in M \times M: q \in \mathbb T_{\gamma_{\text{Fre}}}(r_0), \rho(p,q) \leq r_0\right\}.
\]
Within $S_\tau$ from Lemma \ref{lemma:Lipschitz}, the inverse exponential map $p \mapsto \exp^{-1}_p(\cdot)$ is smooth with bounded first and second derivatives. Thus, 
\[
\sup_{t \in [0,1]}\|\Lambda_2(t)\|_{\fre(t)}=\sup_{t \in [0,1]}\|P^{\beta_t}_{0 \to 1}\exp^{-1}_{\hatfre(t)}x(t)-\exp^{-1}_{\fre(t)}x(t)\|_{\fre(t)} \leq C \sup_{t \in [0,1]}\rho(\hatfre(t), \fre(t)),
\]
and thus of order $O_P(n^{-1/2})$ due to Theorem \ref{thm:C1_convergence} (c). The proof is now complete. 
\subsection{Proof of Proposition \ref{prop:equivariance:in:terms:of:basis}}

The $j$th column of $H(X; \Gamma)$ and $H'(X; \Gamma)$, by Definition \ref{def:unwrap_coordinates}, respectively equal
\begin{align*}
    &U^\top \left[\exp_b^{-1} \{ \gamma(0) \} + P^c_{0 \leftarrow 1} \gamma^\downarrow(t_j) + P^c_{0 \leftarrow 1} P^\gamma_{0 \leftarrow t_j} \exp^{-1}_{\gamma(t_j)}\left\{ x(t_j)\right\}\right]; \\
    & U'^\top \left[\exp_{b'}^{-1} \{ \gamma(0) \} + P^{c'}_{0 \leftarrow 1} \gamma^\downarrow(t_j) + P^{c'}_{0 \leftarrow 1} P^\gamma_{0 \leftarrow t_j} \exp^{-1}_{\gamma(t_j)}\left\{ x(t_j)\right\}\right].
\end{align*}
Direct substitution shows that left-multiplying the former by $A = (U')^\top P^{c'}_{0 \leftarrow 1} P^c_{1\leftarrow 0} (U^\top)^{-1}$ and adding $a = (U')^\top \left[ \exp_{b'}^{-1} \left\{ \gamma(0)\right\}
- P^{c'}_{0 \leftarrow 1} P^c_{1 \leftarrow 0 } (U^\top)^{-1} \exp_b^{-1} \{\gamma(0) \}\right]$ leads to the latter. Hence
$H' (X_i; \Gamma)= a 1_r^\top + A H (X_i; \Gamma) \sim\mathcal{MN}\left(a 1_r^\top +AM_w \Phi, A U_w A^\top, \Phi^\top V_w \Phi \right)$; and the result follows by writing $1_r^\top = 1_k^\top \Phi$, using that the constant function is in the span of $\Phi$. The $a$ and $A$ specific to the result can be computed from the expressions above, with 
$\gamma(0)=\exp_b\{M_w \phi(t_1)\}$. 




\subsection{Proof of Proposition \ref{prop:mhat_fre}}

The solution to \eqref{eqn:M_w_hat:Fre} is
$
\hat{M}^{\text{Fre}}_w=\frac{1}{n}\sum_{i=1}^n{H(X_i ; \hat{\Gamma}_\text{Fre})}\Phi^-. 
$
The $j$th column of $H(X_i; \hat{\Gamma}_\text{Fre})$ is
\begin{equation}
U^\top \left[
{(\hat{\gamma}_\text{Fre})}^\downarrow_b(t_j) + P^c_{0 \leftarrow 1} P^{\hat{\gamma}}_{0 \leftarrow t_j} \exp^{-1}_{\hat{\gamma}_\text{Fre}(t_j)}
\left\{ 
x_i(t_j)\right\}
\right],
\label{eqn:proof:prop:M:hat:Fre:a}
\end{equation}
and by definition of the 
sample Fr\'echet mean, $\hat \gamma_{\text{Fre}}$, for every $j$, 
\begin{equation}
\frac{1}{n}\sum_{i=1}^n \exp^{-1}_{\hat{\gamma}_\text{Fre}(t_j)}\{x_i(t_j)\}=0 \in T_{\hat \gamma_\text{Fre}(t_j)}M.
\label{eqn:proof:prop:M:hat:Fre:b}
\end{equation}
Thus, using \eqref{eqn:proof:prop:M:hat:Fre:a} and \eqref{eqn:proof:prop:M:hat:Fre:b}, and that the parallel transports are linear isometries, we have that
\[
\hat{M}^{\text{Fre}}_w=\frac{1}{n}\sum_{i=1}^n{H(X_i ; \hat{\Gamma}_\text{Fre})}\Phi^-=U^\top(\hat{\Gamma}_\text{Fre})_b^{\downarrow}
\Phi^- =H(\hat{\Gamma}_\text{Fre}; \hat{\Gamma}_\text{Fre})\Phi^-.
\]

\bibliographystyle{plainnat}
\bibliography{paper-ref}

\end{document}